\documentclass [twoside,reqno, 12pt] {amsart}

\usepackage{hyperref}

\usepackage{xcolor}
\definecolor{amethyst}{rgb}{0.6, 0.4, 0.8}

\usepackage{amsfonts}
\usepackage{amssymb}
\usepackage{a4}
\usepackage{color}

\usepackage{graphicx}
\graphicspath{ {./} }

\newtheorem{thm}{Theorem}[section]
\newtheorem{cor}[thm]{Corollary}
\newtheorem{lem}[thm]{Lemma}
\newtheorem{prop}[thm]{Proposition}
\newtheorem{exmp}[thm]{Example}
\newtheorem{defn}[thm]{Definition}
\newtheorem{rem}[thm]{Remark}
\newtheorem{deff}[thm]{Definition}

\theoremstyle{definition}

\numberwithin{equation}{section}

\newtheorem{question}{Question}[]

\newcommand{\C}{\mathbb{C}}

\newcommand{\R}{\mathbb{R}}

\newcommand{\supp}{\operatorname{supp}}

\newcommand{\Z}{\mathbb{Z}}

\newcommand{\ccdot}{\,\cdot\,}
\newcommand{\s}{\hspace{0.5pt}}

\definecolor{armygreen}{rgb}{0.29, 0.33, 0.13}
\definecolor{ao(english)}{rgb}{0.0, 0.5, 0.0}

\newcommand{\abs}[1]{\lvert #1 \rvert} 
\newcommand{\norm}[1]{\lVert #1 \rVert}

\def\hat{\widehat}
\def\tilde{\widetilde}
\def \bfo {\begin {eqnarray*} }
\def \efo {\end {eqnarray*} }
\def \ba {\begin {eqnarray*} }
\def \ea {\end {eqnarray*} }
\def \beq {\begin {eqnarray}}
\def \eeq {\end {eqnarray}}
\def \supp {\hbox{supp }}

\def \dist {\hbox{dist}}

\def \det {\hbox{det}}

\def \p {\partial}

\def\hat{\widehat}
\def\tilde{\widetilde}
\def \bfo {\begin {eqnarray*} }
\def \efo {\end {eqnarray*} }
\def \ba {\begin {eqnarray*} }
\def \ea {\end {eqnarray*} }
\def \beq {\begin {eqnarray}}
\def \eeq {\end {eqnarray}}
\def \supp {\hbox{supp }}

\def \dist {\hbox{dist}}

\def \det {\hbox{det}}

\def \p {\partial}

\begin{document}

 \title[Linearized Calder\'on problem]{Linearized Calder\'on problem and exponentially accurate quasimodes for analytic manifolds}

\author[Krupchyk]{Katya Krupchyk}

\address
        {K. Krupchyk, Department of Mathematics\\
University of California, Irvine\\
CA 92697-3875, USA }

\email{katya.krupchyk@uci.edu}

\author[Liimatainen]{Tony Liimatainen}

\address
       {T. Liimatainen, Department of Mathematics and Statistics\\
       University of Jyv\"as\-ky\-l\"a, Jyv\"askyl\"a\\
       FI-40014, Finland}
\email{tony.t.liimatainen@jyu.fi}

\author[Salo]{Mikko Salo}

\address
       {M. Salo, Department of Mathematics and Statistics\\
       University of Jyv\"askyl\"a, Jyv\"askyl\"a\\
       FI-40014, Finland}
\email{mikko.j.salo@jyu.fi}

\begin{abstract}
% We study the linearized anisotropic Calder\'on problem on a compact Riemannian manifold of dimension $\ge 3$ with boundary which amounts to showing that products of pairs of harmonic functions form a complete set in the space of integrable functions on the manifold.  Assuming that the manifold is transversally anisotropic with the transversal manifold being real analytic and satisfying certain conditions related to the geometry of pairs of intersecting geodesics, we solve the linearized anisotropic Calderon problem.  Our result does not require the assumption of the injectivity of the geodesic X-ray transform on the transversal manifold. Crucial ingredients in the proof of our result are the construction of Gaussian beam quasimodes along non-tangential geodesics on the transversal manifold, with exponentially small errors, as well as the FBI transform characterization of the analytic wave front.

In this article we study the linearized anisotropic Calder\'on problem on a compact Riemannian manifold with boundary. This problem amounts to showing that products of pairs of harmonic functions of the manifold form a complete set. We assume that the manifold is transversally anisotropic and that the transversal manifold is real analytic and satisfies a geometric condition related to the geometry of pairs of intersecting geodesics. In this case, we solve the linearized anisotropic Calder\'on problem. The geometric condition does not involve the injectivity of the geodesic X-ray transform. Crucial ingredients in the proof of our result are the construction of Gaussian beam quasimodes on the transversal manifold, with exponentially small errors, as well as the FBI transform characterization of the analytic wave front set.
% 
% 
% We assume that the manifold is transversally anisotropic and that the transversal manifold being real analytic and satisfies a condition related to the geometry of pairs of intersecting geodesics, we solve the linearized anisotropic Calderon problem.
% 
% 
% We assume that the manifold is transversally anisotropic. Under this assumption, it is known~\cite{DKuLLS_2018} that one can recover transversal singularities from the boundary measurements at points which satisfy a certain geometric condition. The geometric condition does not involve the geodesic X-ray transform. 
%  We extend the results of~\cite{DKuLLS_2018} by additionally assuming that the transversal manifold is analytic.  In this case, we obtain a recovery result of transversal analytic singularities. If additionally all the points satisfy the geometric condition, we in fact obtain a uniqueness result. To prove these results, we construct exponentially accurate quasimodes on general real analytic manifolds with boundary.

\end{abstract}

\maketitle

\section{Introduction and statement of results}

The inverse conductivity problem posed by Calder\'on \cite{Calderon_1980} asks to determine the electrical conductivity of a medium from voltage and current measurements on its boundary. This problem is the mathematical model of Electrical Impedance (or Resistivity) Tomography, an imaging method with applications in seismic and medical imaging. It is also one of the most fundamental models of inverse boundary value problems for elliptic partial differential equations. For these reasons both the theoretical and applied aspects of the Calder\'on problem have been under intense study. We refer to the survey \cite{Uhlmann} for more information and references.

In this article we are interested in the case where the electrical conductivity of the medium is \emph{anisotropic}, i.e.\ depends on direction. This can be modelled by a matrix conductivity coefficient, or in  geometric terms by having a resistivity coefficient given by a Riemannian metric $g$ on a compact manifold $M$ with smooth boundary. There are many variants of this problem. One of them is the (geometric) Calder\'on problem for a Schr\"odinger equation: given a known compact Riemannian manifold $(M,g)$ with smooth boundary and an unknown potential $q \in C^{\infty}(M)$, determine $q$ from the knowledge of the Cauchy data on $\p M$ of solutions of the Schr\"odinger equation 
\[
(-\Delta_g + q)u = 0 \text{ in $M$}.
\]
Here $-\Delta_g$ is the Laplace-Beltrami operator. This geometric  Calder\'on problem is solved in  \cite{GuillarmouTzou_2011} when $\dim(M) = 2$.  The problem is open in general when $\dim(M) \geq 3$ with only partial results available. In particular, the unique determination of $q$ was obtained in \cite{Sylvester_Uhlmann_1987} in the Euclidean setting, in \cite{Isozaki_2004} for hyperbolic manifolds, and in \cite{LeeUhlmann}, \cite{Kohn_Vogelius_1984} in the real analytic setting. Going beyond these settings, the geometric Calder\'on problem was only solved in the case when  $(M,g)$ is CTA (conformally transversally anisotropic, see Definition \ref{def_cta} below) and under the assumption that the geodesic X-ray transform on the transversal manifold is injective  \cite{DKSU_2009, DKuLS_2016}.

The linearized version (at $q=0$) of the above problem is also of interest, since methods for the linearized problem often give insight to the original problem. In our case, the linearized problem reduces to the following simple question asking whether products of pairs of harmonic functions form a complete set in $L^1(M)$:

\begin{question} \label{q_linearized}
Let $(M,g)$ be a compact oriented Riemannian manifold with smooth boundary. If $f \in L^{\infty}(M)$ satisfies 
\[
\int_M f u_1 u_2 \,dV_g = 0
\]
for all $u_j \in L^2(M)$ with $\Delta_g u_j = 0$ in $M$, $j=1,2$, is it true that $f \equiv 0$?
\end{question}

The methods of \cite{GuillarmouTzou_2011, DKSU_2009, DKuLS_2016} give a positive answer to Question \ref{q_linearized} when $\dim(M) = 2$, or when $\dim(M) \geq 3$ and $(M,g)$ is CTA with the transversal manifold having injective geodesic X-ray transform. There have been recent attempts to improve these results when $\dim(M) \geq 3$. In \cite{GuillarmouSaloTzou_2019}, it is proved that Question \ref{q_linearized} has a positive answer when $(M,g)$ is a complex K\"ahler manifold with sufficiently many holomorphic functions. The article \cite{DKuLLS_2018} establishes a recovery of singularities result: if $(M,g)$ is transversally anisotropic and the transversal manifold satisfies a certain geometric condition, one can recover transversal singularities of $f$. In a related work~\cite{LLLS19}, it is proved that on a general transversally anisotropic manifold products of sets of four (instead of pairs) of harmonic functions form a complete set in $L^1(M)$. See also \cite{DKSjU_2009},  \cite{Sjostrand_Uhlmann_2016} for  the linearized Calder\'on problem with partial data in the Euclidean setting. 

In this article we extend the result of \cite{DKuLLS_2018} and show that if the transversal manifold is additionally real-analytic, Question \ref{q_linearized} has a positive answer (i.e.\ one can recover $f\in L^\infty(M)$ completely, not just some of its singularities). 

Let us proceed to state our results. To that end, let us first  recall the following definitions, see \cite{DKSU_2009}, \cite{DKuLS_2016}. 
\begin{defn} \label{def_cta}
Let $(M,g)$ be a smooth compact oriented Riemannian manifold of dimension $n\ge 3$ with smooth boundary $\p M$.
\begin{itemize}
\item[(i)] $(M,g)$ is called transversally anisotropic if $(M,g)\subset \subset (T,g)$ where $T=\R\times M_0^{\mathrm{int}}$, $g=e\oplus g_0$, $(\R,e)$ is the Euclidean real line, and $(M_0,g_0)$ is a smooth compact $(n-1)$--dimensional manifold with smooth boundary, called the transversal manifold. 

\item[(ii)] $(M,g)$ is called conformally transversally anisotropic (CTA) if $(M,cg)$ is transversally anisotropic, for some  positive function $c\in C^\infty(M)$. 

%\item[(iii)] $(M,g)$ is called admissible if $(M,g)$ is CTA and the transversal manifold $(M_0,g_0)$ is simple, meaning that for any $p\in M_0$, the exponential map $\exp_p$ with its maximal domain of definition in $T_p M_0$ is a diffeomorphism onto $M_0$, and $\p M_0$ is strictly convex.
\end{itemize}
\end{defn}
Here and in what follows  $M_0^{\text{int}}=M_0\setminus\p M_0$ stands for the interior of $M_0$.

Let $(M,g)$ be transversally anisotropic of dimension $n\ge 3$ with a transversal manifold $(M_0,g_0)$. Next we need some definitions related to the transversal manifold $(M_0,g_0)$.  Following \cite{DKuLS_2016}, we say that a geodesic $\gamma:[-T_1,T_2]\to M_0$, $0<T_1,T_2<\infty$,  is nontangential if $\gamma(-T_1), \gamma(T_2)\in \p M_0$, $\gamma(t)\in M_0^{\text{int}}$ for all $-T_1<t<T_2$, and $\dot{\gamma}(-T_1)$, $\dot{\gamma}(T_2)$ are nontangential vectors on $\p M_0$.  Following \cite{DKuLLS_2018}, we have the following definition. 
\begin{defn} 
\label{def_admissible}
We say that  $(x_0',\xi_0')\in S^*M_0^{\mathrm{int}}$ is generated by an admissible pair of geodesics, if 
there are two nontangential unit speed geodesics 
\[
\gamma_1:[-T_1,T_2]\to M_0, \quad \gamma_2:[-S_1,S_2]\to M_0,
\] 
$0<T_1,T_2, S_1,S_2<\infty$, such that 
\begin{itemize}
\item[(i)] $\gamma_1(0)=\gamma_2(0)=x_0'$,
\item[(ii)] $\dot{\gamma}_1(0)+\dot{\gamma}_2(0)=t_0\xi_0'$,
for some $0<t_0<2$, where $\xi_0'$ is understood as an element of $T_{x_0}M_0^{\mathrm{int}}$ by the Riemannian  duality, 
\item[(iii)] $\gamma_1$, $\gamma_2$ do not have self-intersections at the point $x_0'$, and $x_0'$ is the only point of their intersections, i.e.  
\begin{align*}
\gamma_1(t)=x_0' \Leftrightarrow t=0,\quad \gamma_2(s)=x_0' \Leftrightarrow s=0,\\
\gamma_1(t)=\gamma_2(s) \Rightarrow \gamma_1(t)=\gamma_2(s)=x_0'.
\end{align*}
\end{itemize}
\end{defn}

Let $f\in L^\infty(M)$ and let us extend $f\in L^\infty(M)$ by zero to $(\R\times M_0)\setminus M$. Writing $x=(x_1,x')$ where $x_1\in \R$, and $x'$ are local coordinates $M_0$,  we let 
\[
\hat f(\lambda, x')=\int_{-\infty}^\infty e^{-i\lambda x_1} f(x_1,x') \,dx_1, \quad \lambda\in \R, 
\]
be the Fourier transform of $f$ with respect to $x_1$.  We have for each  $\lambda\in \R$ that $\hat f(\lambda, \ccdot)\in L^\infty(M_0)\cap \mathcal{E}'(M_0^{\mathrm{int}})$.

When $X$ is a real analytic open manifold and $u\in \mathcal{D}'(X)$, we let $WF_a(u) \subset T^*X\setminus\{0\}$ stand for the analytic wave front set of $u$, see \cite[Definition 6.1]{Sjostrand_Analytic}, \cite[Sections 8.5, 9.3]{Hormander_bookI}. The set $WF_a(u)\subset T^*X \setminus\{0\}$ is closed conic and we have 
\[
\pi(WF_a(u))=\text{singsupp}_a(u),
\]
where $\pi: T^*X\to X$, $(x,\xi)\mapsto x$, is the natural projection and $\text{singsupp}_a(u)$ is the analytic singular support of $u$, i.e. the smallest closed set such that $u$ is real analytic in the complement. In particular, 
$WF_a(u)=\emptyset$ if and only if $u$ is real analytic on $X$.

We have the following analytic microlocal result, which is an analog of Theorem 1.1 in \cite{DKuLLS_2018}, established in the $C^{\infty}$--case. 
\begin{thm}
\label{thm_main_2}
Let $(M,g)$ be a transversally anisotropic manifold of dimension $n\ge 3$ with transversal manifold $(M_0,g_0)$, and assume that $M_0^{\mathrm{int}}$ and $g_0|_{M_0^{\mathrm{int}}}$ are real analytic. Assume furthermore  that $f\in L^\infty(M)$ satisfies
\begin{equation}
\label{eq_int_orthog}
\int_M fu_1 u_2 \,dV_g=0,
\end{equation}
for all $u_j\in L^2(M)$ with $-\Delta_g u_j=0$ in $M^{\mathrm{int}}$. Let $(x'_0,\xi'_0)\in S^*M_0^{\mathrm{int}}$ be generated by an admissible pair of geodesics. Then
for any $\lambda\in \R$, one has 
\[
(x'_0,\xi'_0)\notin WF_a(\hat f(\lambda,\ccdot) )\subset T^*M_0^{\mathrm{int}}\setminus\{0\}. 
\]
\end{thm}

Theorem \ref{thm_main_2} implies the following global result, which gives a positive answer to Question \ref{q_linearized} under suitable  geometric assumptions.
 
\begin{thm}
\label{thm_main}
Let $(M,g)$ be a transversally anisotropic manifold of dimension $n\ge 3$  and assume that the transversal manifold $(M_0,g_0)$ is connected, $M_0^{\mathrm{int}}$ as well as  $g_0$ in $M_0^{\mathrm{int}}$ are real analytic.  Assume that every point $(x'_0,\xi'_0)\in S^*M_0^{\mathrm{int}}$ is generated by an admissible pair of geodesics. 
Moreover, assume that $f\in L^\infty(M)$ satisfies  \eqref{eq_int_orthog} 
for all $u_j\in L^2(M)$ with $-\Delta_g u_j=0$ in $M^{\mathrm{int}}$.  Then $f=0$ in $M$. 
\end{thm}

\begin{rem}
Note that while $(M_0^{\mathrm{int}},g_0)$ is real analytic, Theorem \ref{thm_main} does not follow from the existing results in the real analytic setting, as it corresponds to deforming the zero potential by an $L^\infty$ perturbation. 
\end{rem}

\begin{rem}
In Theorems \ref{thm_main_2} and \ref{thm_main} while $M^{\mathrm{int}}$ is real analytic, the boundary $\p M$ need not be real analytic.  
\end{rem}

As  the following example shows, there exist transversally anisotropic manifolds $(M,g)$ with a transversal manifold $(M_0,g_0)$ satisfying the geometric conditions of Theorem \ref{thm_main} and with a non-invertible geodesic X-ray transform. Therefore, the geometric Calder\'on problem is still open on such manifolds while our Theorem \ref{thm_main} gives a positive solution to the corresponding linearized problem.

\begin{exmp}
\label{ex_main}
Let $M_0 = \mathbb{S}^1 \times [0,a]$, $a > 0$, be a cylinder with its usual flat metric $g_0$.  The geodesics on $M_0$ are straight lines, circular cross sections, and helices that wind around the cylinder. The geodesic X-ray transform
 is not invertible, since the kernel contains functions of the form $f(e^{it}, s) = h(s)$ where $h \in C^{\infty}_0((0,a))$ integrates to zero over $[0,a]$.  However, it is shown in Appendix \ref{app_example} that every point $(x'_0,\xi'_0)\in S^*M_0^{\mathrm{int}}$ is generated by an admissible pair of geodesics. 
\end{exmp}

It is established in  \cite[Lemma 3.1]{DKuLLS_2018} that if $(M_0,g_0)$ satisfies the strict Stefanov--Uhlmann regularity condition at $(x'_0,\xi'_0)\in S^*M_0^{\mathrm{int}}$, which we now proceed to recall, then $(x'_0,\xi'_0)$ is generated by an admissible pair of geodesics.  
\begin{defn}
The transversal manifold $(M_0,g_0)$ satisfies  the strict Stefanov--Uhlmann regularity condition at $(x_0',\xi_0')\in S^*M_0^{\mathrm{int}}$ if there exists $\eta'\in S_{x'_0}^*M_0^{\mathrm{int}}$ such that ${g_0}(\xi_0',\eta')=0$ and such that the following holds: let $\gamma_{x'_0,\eta'}:[-T_1,T_2]\to M_0$, $0<T_1,T_2<\infty$, be the geodesic with $\gamma_{x'_0,\eta'}(0)=x'_0$, $\dot{\gamma}_{x'_0,\eta'}=\eta'$. We have 
\begin{itemize}
\item[(i)] $\gamma_{x'_0,\eta'}$ is nontangential,
\item[(ii)] $\gamma_{x'_0,\eta'}$ contains no points conjugate to $x'_0$,
\item[(iii)] $\gamma_{x'_0,\eta'}$ does not self-intersect for any time $t\in [-T_1,T_2]$.  
\end{itemize}
 \end{defn}
Hence, if  a transversally anisotropic manifold $(M,g)$ is such that the transversal manifold $(M_0,g_0)$ satisfies the strict Stefanov--Uhlmann regularity condition at every point of $S^*M_0^{\mathrm{int}}$ with $M_0^{\mathrm{int}}$ and $g_0|_{M_0^{\mathrm{int}}}$ real analytic, and $(M_0,g_0)$ is connected, then Theorem \ref{thm_main} holds. 

As the following examples demonstrate, there are transversally anisotropic manifolds $(M,g)$ with a transversal manifold $(M_0,g_0)$ satisfying the geometric condition of Theorem \ref{thm_main}, and with an invertible geodesic X-ray transform. Thus, for such manifolds  $(M,g)$, Theorem \ref{thm_main} also follows from \cite{DKSU_2009}, \cite{DKuLS_2016}. 

\begin{exmp}
Let $(M_0,g_0)$ be a \emph{simple} manifold, i.e.\ a compact simply connected manifold with strictly convex boundary so that no geodesic has conjugate points. Then $(M_0,g_0)$ satisfies the strict Stefanov--Uhlmann regularity condition at any point of $S^*M_0^{\mathrm{int}}$ and thus also the geometric condition in Theorem \ref{thm_main}. Note that in this case $(M,g)$ is admissible in the sense of \cite{DKSU_2009}, and Theorem \ref{thm_main} would also follow from \cite{DKSU_2009}.
\end{exmp}

\begin{exmp}
Let $\mathbb{S}^3\subset \R^4$ be the unit sphere and let $\mu$ be a geodesic arc from the north pole to the south pole of the sphere. Let $M_0$ be the closure of a neighborhood of $\mu$. It is established in \cite{DKuLLS_2018}  that the manifold $M_0$ satisfies the strict Stefanov--Uhlmann regularity condition at each point of $ S^*M_0^\mathrm{int}$. Notice also that the manifold $M_0$ contains conjugate points, so that it is not simple. However, the geodesic X-ray transform on $(M_0,g_0)$ is injective by \cite{Stefanov_Uhlmann_2008}, and Theorem \ref{thm_main} would therefore also follow from \cite{DKuLS_2016}.
\end{exmp}

\begin{rem} We would like to remark that the strict Stefanov-Uhlmann condition is not satisfied for $(M_0,g_0)$ of Example \ref{ex_main} since for any $(x_0', \xi_0') \in S^* M_0^{\mathrm{int}}$ with $\xi_0'$ pointing in the direction of the $[0,a]$ factor, the orthogonal geodesics never reach $\p M_0$. 
\end{rem}

The proof of Theorem \ref{thm_main_2} depends crucially on the construction of Gaussian beam quasimodes along nontangential geodesics on $M_0$, with exponentially small errors, as stated in the following result. Before stating the result, let us recall from \cite[Chapter 1]{Sjostrand_Analytic} the notion of a classical analytic symbol. Let $V \subset \C^n$ be an open set. We say that $a(x;h)=\sum_{k=0}^\infty h^k a_k(x)$ is a (formal) classical analytic symbol in $V$  if  $a_k\in \text{Hol}(V)$, $k=0,1,2,\dots$, and for every $\tilde V\subset\subset V$, there exists $C=C_{\tilde V}>0$ such that 
\begin{equation}
\label{eq_4_2}
|a_k(x)|\le C^{k+1}k^k, \quad x\in \tilde  V, 
\end{equation}
$ k=0, 1,2,\dots$. The classical analytic symbol $a(x;h)$ is said to be elliptic if $a_0\ne 0$. 

We have the following essentially well known result,  see \cite{Sjostrand_1975} and \cite{Sjostrand_Analytic}, and see also  \cite{Babich_96} for a sketch of the proof.  Notice that here our quasimode construction is performed along the entire geodesic segment contrary to the standard constructions in a neighborhood of a point, see \cite{Dencker_Sj_Zworski}.
\begin{thm}
\label{thm_main_1}
Let $(X,g)$ be a compact Riemannian manifold of dimension $n\ge 2$ with smooth boundary, contained in a real analytic open manifold $(\hat X,g)$ of the same dimension with $g$ real analytic in $\hat X$.
Let 
$\gamma:[-T_1,T_2]\to X$, $0<T_1,T_2<\infty$,  be a unit speed non-tangential geodesic in $X$, and let $\lambda\in \R$. There is a family of $C^\infty$  functions $v(x;h)$ on $X$, $0<h\le 1$, and $C>0$ such that $\supp(v(\ccdot;h))$ is confined to a small neighborhood of $\gamma([-T_1,T_2])$ and 
\begin{equation}
\label{eq_prop_gaussian_1}
\| (-h^2\Delta_g-(hs)^2)v\|_{L^2(X)}=\mathcal{O}(e^{-\frac{1}{Ch}}), \quad \|v\|_{L^2(X)}\asymp 1, 
\end{equation}
as $h\to 0$.  Here $s=\frac{1}{h}+i\lambda$. The local structure of the family $v(x\s ;h)$ is as follows: let $p\in \gamma([-T_1, T_2])$ and let $t_1<\dots<t_{N_p}$ be the times in $(-T_1,T_2)$ when $\gamma(t_l)=p$, $l=1,\dots, N_p$. In a sufficiently small neighborhood $V$ of a point $p\in \gamma([-T_1, T_2])$, we have 
\[
v|_{V}=v^{(1)}+\dots+v^{(N_p)}, 
\]
where each $v^{(l)}$ has the form 
\[
v^{(l)}(x;h)= h^{-\frac{(n-1)}{4}}e^{is \varphi^{(l)}(x)}a^{(l)}(x;h).
\]
Here $\varphi=\varphi^{(l)}$ is real analytic in $V$ satisfying for $t$ near $t_l$, 
\begin{equation}
\label{eq_prop_gaussian_1_phase_100}
\varphi(\gamma(t))=t, \quad \nabla \varphi(\gamma(t))=\dot{\gamma}(t), \quad \emph{\text{Im}}\, (\nabla^2\varphi(\gamma(t)))\ge 0, \quad \emph{\text{Im}}\, (\nabla^2\varphi)|_{ \dot{\gamma}(t)^\perp}> 0, 
\end{equation}
and $a^{(l)}$ is an elliptic classical analytic symbol in a complex neighborhood of $p$. 
\end{thm}

We have chosen to give a fairly complete proof of Theorem  \ref{thm_main_1} since we are not aware of a detailed treatment in the literature and since we need to have fairly precise information concerning the quasimodes for our applications. 

%Let us shortly explain how the exponentially small error is achieved in the theorem. 
Let us briefly mention how the exponentially small error is achieved in Theorem \ref{thm_main_1}. 
The proof of the theorem is by using the ansatz $v(x;h)=e^{is \varphi(x)}a(x;h)$, which, as usual, leads to solving the eikonal equation for the phase function $\varphi(x)$ and a transport equation for the amplitude $a(x;h)$. We first find an exact analytic solution for the eikonal equation near a geodesic segment of $\gamma$. Consequently, the transport equation for the amplitude $a(x;h)=\sum_{k=0}^N h^k a_k(x)$ has analytic coefficients and we find $a(x;h)$ as a classical analytic symbol. This involves adapting the nested neighborhood method of \cite{Sjostrand_Analytic}. 
%We find $a(x;h)$ as a solution to the transport equation, which is a classical analytic symbol. 
The error term for $v$ being a true eigenfunction then is 
\begin{equation}
\label{eq:error_term_intro}
(-h^2\Delta_g-(hs)^2)e^{is\varphi}\bigg(\sum_{j=0}^Nh^j a_j\bigg)=h^{N+2} T_2(a_N),
\end{equation}
% 
% $(-h^2\Delta_g-(hs)^2)v$ for the ansatz of $v$ being a true eigenfunction is of the form
% \begin{equation}\label{eq:error_term_intro}
%  h^{N+2}\s T_2(a_N), %a_N(\Delta_g a_N+\lambda (L_0+\Delta_g\varphi)a_N),
% \end{equation}
where $T_2$ is a second  order operator with analytic coefficients.
Cauchy estimates and~\eqref{eq_4_2} then yield that the error term~\eqref{eq:error_term_intro} is bounded by $h^{N+2}C^{N+1}N^N$. Letting the order $N$ of the expansions of $a$ depend on $h$ as $N=N(h)=[\frac{1}{heC}]$ gives the exponentially small error in the theorem.

The above was based on finding first an exact analytic solution to the eikonal equation $\abs{d\varphi}_g=1$ near a geodesic segment of $\gamma$. To find such a solution, we view the eikonal equation as the Hamilton-Jacobi equation,
\begin{equation}\label{eq:HJ}
 p(x,\varphi_x'(x))=0,
\end{equation}
where $p(x,\xi)=\abs{\xi}_{g(x)}^2-1$ is holomorphically continued to a complex domain.
%Here $T^*X$ is the complexified cotangent bundle. 
When solving the Hamilton-Jacobi equation \eqref{eq:HJ} we proceed by a geometric argument of constructing a complex Lagrangian manifold,
\[
 \Lambda \subset p^{-1}(0),
\]
in a complex neighborhood of a segment of the graph of $\dot\gamma\subset T^*X$, see \cite{Sjostrand_1975}.
% 
% 
% We solve the Hamilton-Jacobi equation by constructing an analytic Lagrangian submanifold $\Lambda\subset T^*X$ inside $p^{-1}(0)$,
% \[
%  \Lambda \subset p^{-1}(0)\subset T^*X.
% \]
The solution $\varphi$ is then obtained as a generating function of the Lagrangian $\Lambda$, which parametrises $\Lambda$ as
\[
 \Lambda=\{(x,\varphi'_x(x)\}.
\] 
%The Lagrangian $\Lambda$ then yields a solution $\varphi$ to the eikonal equation 
% A Lagrangian can be (locally) 
% parametrized as the graph of the gradient of a function $\varphi:X\to \C$,
% \[
%  \Lambda=\{(x,\varphi'_x(x)\}.
% \]
Extending the argument to a neighborhood of the geodesic segment of $\gamma$ requires some extra work involving positive Lagrangians.

Let us proceed to explain the main ideas in the proof of Theorem \ref{thm_main_2}. Let $\alpha_0=(x_0',\xi_0')\in S^*M_0^{\text{int}}$ be generated by an admissible pair of geodesics $\gamma_1(\alpha_0)$ and $\gamma_2(\alpha_0)$ on $M_0$. We first show that there exists a neighborhood of $\alpha_0$ in $S^*M_0^{\text{int}}$ such that every point $\alpha$ in the neighborhood is generated by an admissible pair of geodesics $\gamma_1(\alpha)$ and $\gamma_2(\alpha)$ on $M_0$. Next we construct two real analytic families of Gaussian beams quasimodes $v_1(\alpha)$ and $v_2(\alpha)$ on $M_0$, associated to $\gamma_1(\alpha)$ and $\gamma_2(\alpha)$, respectively, with exponentially small errors. The fact that $(M,g)$ is transversally anisotropic provides us with the limiting Carleman weight $\phi(x)=x_1$ for the Laplacian, and using the technique of Carleman estimates, we convert the families of Gaussian beams $v_1(\alpha)$ and $v_2(\alpha)$ into two families of harmonic functions on $M$ with exponentially small remainder terms. Testing the orthogonality relation \eqref{eq_int_orthog} with the constructed families of harmonic functions leads to the exponential decay of the FBI transform of $\hat f(\lambda,\cdot)$ in a neighborhood of $\alpha_0$. Using the FBI characterization of the analytic wave front set, see \cite{Sjostrand_Analytic}, we conclude the proof. Note that we need to work with families of Gaussian beams to fill out the entire neighborhood of $\alpha_0$.  

\begin{rem}
Similarly to \cite{DKuLLS_2018}, Theorem \ref{thm_main_2} and Theorem \ref{thm_main} are established for transversally anisotropic manifolds rather than CTA manifolds.  The reason for this is that the standard reduction 
\[
c^{\frac{n+2}{4}} \circ (-\Delta_{cg}) \circ c^{-\frac{(n-2)}{4}} = -\Delta_{g}+ q, \quad  g=e\oplus g_0,
\]
leads to the potential 
\[
 q=  - c^{\frac{n+2}{4}}\Delta_g(c^{-\frac{(n-2)}{4}}),
\]
see \cite{DKuLS_2016}, and therefore, to construct harmonic functions with exponentially small remainder terms on a CTA manifold, one has to construct Gaussian beam quasimodes for the conjugated Schr\"odinger operator,
\[
e^{sx_1}(-h^2\Delta_g+h^2q)e^{-sx_1}, 
\]
with exponentially small errors. If $c$ is independent of $x_1$ and real-analytic then so is $q$, and this construction could be done as in  Theorem \ref{thm_main_1}. Notice also that for this reason, one can also include a general real analytic potential which is independent of $x_1$ in the results of   Theorem \ref{thm_main_1}. However, if $c$ depends on $x_1$ then the corresponding sequence of transport equations becomes of $\bar \p$-type, see e.g. \cite{Feizmohammadi_Oksanen_2020}, \cite{Krup_Uhlmann_2018}, which complicates the analysis of Theorem \ref{thm_main_1} further and and is therefore not developed here. 
\end{rem}

Let us mention that Gaussian beam quasimode constructions have a long tradition in microlocal analysis, see  \cite{Babich_Buldyrev}, \cite{Ralston_1977}, \cite{Ralston_1982}, with applications in the analysis of eigenfunctions, see   \cite{Zelditch}, and inverse problems, see   \cite{Oksanen_Salo_Stefanov_Uhlmann_2020} and the references given there.  

Finally, let us point out certain related results on a standard geometric version of the Calder\'on problem, which asks to determine a metric $g$ up to natural gauges (a boundary-fixing diffeomorphism, and also a conformal factor when $\dim(M) = 2$) from the knowledge of Cauchy data on $\p M$ of solutions of the equation $-\Delta_g u = 0$ in $M$.  This problem was solved in   \cite{LassasUhlmann} when $\dim(M) = 2$, but for $\dim(M) \geq 3$ it is only known under additional conditions such as the manifold being real-analytic, see \cite{LeeUhlmann, LassasUhlmann, LassasTaylorUhlmann}, or Einstein \cite{GuillarmouSaBarreto}. Alternative proofs are given in \cite{Belishev_2009, LassasLiimatainenSalo_2019}. Interesting counterexamples in the case of measurements on disjoint sets or low regularity coefficients are given in \cite{DaudeKamranNicoleau_2019, DaudeKamranNicoleau_2020}. If one allows degenerate coefficients, there are other counterexamples~\cite{LassasTaylorUhlmann, GLU3}. Counterexamples with degenerate coefficients form the basis of invisibility cloaking, see e.g.~\cite{Uhlmann}.

The paper is organized as follows.  Section \ref{sec_quasimodes} is devoted to the construction of exponentially accurate Gaussian beam quasimodes and the proof of Theorem~\ref{thm_main_1}.  Section \ref{sec_geodesics} contains some results concerning properties of geodesics needed in the proof of Theorem \ref{thm_main_2}.  Section \ref{sec_analytic_family_quasimodes} extends Theorem  \ref{thm_main_1} to produce  analytic families of exponentially accurate Gaussian beam quasimodes. 
The construction of families of harmonic functions based on Gaussian beam quasimodes is presented in Section 
\ref{sec_CGO_based_quasi}. Section \ref{sec_proofs_of_main_theorems} contains some facts about analytic wave front sets and the proofs of Theorem \ref{thm_main_2} and Theorem \ref{thm_main}.  The admissibility property of geodesics in Example \ref{ex_main} is verified in Appendix \ref{app_example}.

\subsection*{Acknowledgements}

The research of K.K.\ is partially supported by the National Science Foundation (DMS 1815922). T.L.\ and M.S.\ were supported by the Academy of Finland (Finnish Centre of Excellence in Inverse Modelling and Imaging, grant numbers 312121 and 309963), and M.S.\ was also supported by the European Research Council under Horizon 2020 (ERC CoG 770924). This material is based upon work supported by the National Science Foundation under Grant No.\ 1440140, while K.K.\ and M.S.\ were in residence at MSRI in Berkeley, California, during the semester on Microlocal Analysis in 2019. K.K.\ is very grateful to Johannes Sj\"ostrand for a very helpful discussion and for pointing out the reference \cite{Sjostrand_1975}.

\section{Exponentially accurate quasimodes. Proof of Theorem \ref{thm_main_1}}\label{sec:exp_accurate_quasimodes}

\label{sec_quasimodes}

Let $(X,g)$ be a compact Riemannian manifold of dimension $n\ge 2$ with smooth boundary, contained in a larger real analytic open manifold $(\hat X,g)$ of the same dimension with $g$ real analytic in $\hat X$. We extend $\gamma$ as a unit speed geodesic in $\hat X$.  Let $\varepsilon>0$ be such that $\gamma(t)\in 
\hat{X}\setminus X$ and $\gamma(t)$ has no self-intersection for $t\in [-T_1-2\varepsilon, -T_1)\cup (T_2,T_2+2\varepsilon]$. This choice of $\varepsilon$ is possible since $\gamma$ is non-tangential.  First  it follows from \cite[Lemma 7.2]{Kenig_Salo_APDE_2013} that $\gamma|_{[-T_1-\varepsilon,T_2+\varepsilon]}$ self-intersects only at finitely many times $t_j$ with 
\[
-T_1< t_1<\dots <t_N< T_2.
\]
We also set $t_0:=-T_1-\varepsilon$ and $t_{N+1}:=T_2+\varepsilon$. An application of  \cite[Lemma 3.5]{DKuLS_2016} shows that there exists an open cover $\{(U_j,\kappa_j)\}_{j=0}^{N+1}$ of $\gamma([-T_1-\varepsilon,T_2+\varepsilon])$ consisting of coordinate neighborhoods $U_j$ and real analytic diffeomorphisms $\kappa_j$ having the following properties: 
\begin{itemize}
\item[(i)] $\kappa_j(U_j)=I_j\times B$, where $I_j$ are open intervals and $B=B(0,\delta')$ is an open ball in $\R^{n-1}$.  Here $\delta'>0$  can be taken arbitrarily small and the same for each $U_j$, 
\item[(ii)] $\kappa_j(\gamma(t))=(t,0)$ for each $t\in I_j$,
\item[(iii)] $t_j$ only belongs to $I_j$ and $\overline{I_j}\cap \overline{I_k}=\emptyset$ unless $|j-k|\le 1$,
\item[(iv)] $\kappa_j=\kappa_k$ on $\kappa_j^{-1}((I_j\cap I_k)\times B)$.
\end{itemize}
The corresponding local coordinates $\kappa_j(x)=(t,y)\in U_j$ are called the Fermi coordinates. Here we note that  Lemma 3.5 in  \cite{DKuLS_2016}  is established in the $C^{\infty}$ case, and the real analyticity of the Fermi diffeomorphisms $\kappa_j$ is obtained by inspection of the proof of Lemma 3.5 in  \cite{DKuLS_2016}, in view of the analyticity of $\hat X$. 
As observed in the proof of \cite[Lemma 3.5]{DKuLS_2016}, in the case when $\gamma$ does not self-intersect, there are Fermi coordinates on a single coordinate neighborhood of $\gamma|_{[-T_1-\varepsilon, T_2+\varepsilon]}$ so that (i) and (ii) are satisfied. These coordinates are given by inverting the map
\[
 (t,y)\mapsto \text{exp}_{\gamma(t)}\big(\sum_{k=1}^{n-1}y^k\s e_k(t)\big) \in \hat X.
\]
Here $e_k(t)$ are the parallel transportations of the last $n-1$ vectors of  an orthonormal frame $\{\dot{\gamma}|_{t=-T_1},e_1,\ldots,e_{n-1}\}\subset T_{\gamma(-T_1)}M$ and $\exp$ is the exponential map of $(\hat X,g)$.

Our goal is to construct exponentially accurate Gaussian beam quasimodes near $\gamma([-T_1-\varepsilon, T_2+\varepsilon])$. We shall start by carrying out the quasimode construction in a fixed coordinate neighborhood $U=U_j$ which we can identify with the set 
$I\times B$, where $I\subset \R$ is an open interval and $B=B(0,\delta')$ is an open ball in $\R^{n-1}$ with $\delta'>0$. Without loss of generality, we assume that $0\in I$. 
The geodesic $\gamma$ in the open set $U$ is given by $\Gamma=\{x=(t,y)\in I\times B:y=0\}$. 

Let us consider the following Gaussian beam ansatz, 
\begin{equation}
\label{eq_3_0_quasimode}
v(t,y;h)=e^{is \varphi(t,y)}a(t,y;h), \quad s=\frac{1}{h}+i\lambda,\quad \lambda\in \R,
\end{equation}
where the phase $\varphi$ is complex valued with $\text{Im}\, \varphi(t,y)\ge 0$ and $a$ is an amplitude. We shall  proceed to construct the quasimode $v$ so that the phase $\varphi$ satisfies \eqref{eq_prop_gaussian_1_phase_100} and the amplitude $a$ is an elliptic classical analytic symbol.  

\subsection{Construction of the phase function $\varphi$}  We shall proceed using the classical arguments, solving the Hamilton-Jacobi equation in the complex domain and making crucial use of the geometry of positive complex Lagrangians, see \cite{Sjostrand_1975}. Let us remark here that while we only need the good properties of the phase in the real domain, specifically along the geodesic $\gamma$, since the phase function takes complex values, the Hamilton-Jacobi equation holds naturally for the holomorphic extensions in the complex domain. From the geometric point of view, the complex Lagrangian manifold naturally associated to the phase function $\varphi$ is not confined to the real domain but is a submanifold of the complexified phase space.

First we have  
\begin{equation}
\label{eq_3_0}
\begin{aligned}
e^{-is \varphi}(-h^2\Delta_g-(hs)^2)e^{is \varphi}a=-h^2\Delta_g a-ih(1+i\lambda h)[2\langle d\varphi, da\rangle_g +(\Delta_g \varphi)a]\\
+(1+i\lambda h)^2[\langle d\varphi, d\varphi\rangle_g-1]a.
\end{aligned}
\end{equation}
In the usual Gaussian beam construction in the $C^{\infty}$--setting, one solves the eikonal equation to a large, and sometimes infinite, order along the geodesic, see \cite{Oksanen_Salo_Stefanov_Uhlmann_2020}, \cite{Babich_Buldyrev}, \cite{Ralston_1977}, \cite{Ralston_1982}. Working in the present real analytic setting, it will be natural to 
solve the eikonal equation 
\begin{equation}
\label{eq_3_1}
\langle d\varphi, d\varphi\rangle_g-1=p(x,\varphi'_x(x))=0
\end{equation}
in a full neighborhood of the geodesic. Here 
\begin{equation}
\label{eq_3_1_2}
p(x,\xi)=|\xi|_g^2-1=G(x) \xi \cdot \xi - 1
\end{equation}
 is the semiclassical principal symbol of the operator  $P=-h^2\Delta_g-(hs)^2$, where $G(x)=(g^{jk}(x))$.  Since the metric $g$ is real analytic, $p(x,\xi)$ extends to a holomorphic function in an open set of the form $\tilde U\times \C^n$, where 
 \[
\tilde U\subset \C^n  
 \]
is a complex neighborhood of $U$.   

Let $(x(t), \xi(t))=\exp (\frac{t}{2}H_p)(0,\xi_0)$ be the integral curve of the Hamiltonian $H_p$ in $T^*X$, which corresponds to the unit speed geodesic $\gamma$, so that 
\[
\pi_x\bigg(\exp \bigg(\frac{t}{2}H_p\bigg)(0,\xi_0)\bigg)=\gamma(t), \quad t\in I\subset \R, 
\]
where $\pi_x(x,\xi)=x$, and $\xi_0=\dot{\gamma}(0)$. Here $\dot{\gamma}(0)$ is viewed as a cotangent vector using the Riemannian duality.  Since $(0, \xi_0)\in p^{-1}(0)\cap(U\times\R^n)$, we therefore have $(x(t), \xi(t))\in p^{-1}(0)\cap(U\times\R^n)$ for all $t\in I$. We have explicitly the Hamilton's equations
\begin{equation}
\label{eq_3_2}
\begin{cases}
\dot{x}(t)=\frac{1}{2}\p_{\xi}p(x(t),\xi(t)),\\
\dot{\xi}(t)=-\frac{1}{2}\p_x p(x(t),\xi(t)),\\
x(0)=0,\\
\xi(0)=\xi_0.
\end{cases}
\end{equation}
Recalling that $x=(t,y)\in U$ and writing $\xi=(\tau, \eta)\in T^*_xX$ for the dual variable, we see from \eqref{eq_3_2} that 
\begin{equation}
\label{eq_3_2_new_1}
\p_\tau p(x(t),\xi(t))\ne 0\quad \text{for all } t \in I,
\end{equation} 
since the $t$ component of $\dot x(t)$ is identically $1$ in the $(t,y)$ coordinates.

We look for a real analytic solution $\varphi$ of \eqref{eq_3_1} in $U$ such that 
\begin{equation}
\label{eq_3_4}
\text{Im}\,  \varphi (t,y)\ge 0, \quad \text{Im}\, \varphi(t,0) =0, \quad  \text{Im}\, \varphi''_{yy}(t,0)>0,\quad t\in I, 
\end{equation}
and therefore, 
\[
\text{Im}\, \varphi(t,y)\sim |y|^2=\dist((t,y), \Gamma)^2, \quad  (t,y)\in U.
\]
We will find the required real analytic solution of \eqref{eq_3_1} as the restriction to $U\subset \R^n$ of a holomorphic function $\varphi$ in $\tilde U\subset \C^n$, solving the following Cauchy problem for the Hamilton-Jacobi equation in the complex domain,
\begin{equation}
\label{eq_3_3}
\begin{cases} p(x,\varphi'_x(x))=0,\quad x=(t,y)\in \tilde U,\\
\varphi(0, y)=\psi(y),\\
\varphi'_x(0)=\xi_0.
\end{cases}
\end{equation}
Here we take $\psi$ to be  a holomorphic function near $0\in \C^{n-1}$ such that 
\begin{equation}
\label{eq_3_4_1}
\text{Im}\, \psi''_{yy}(0)>0, 
\end{equation}
$\psi(0)$ is real, 
and so that the compatibility condition  $\psi_y'(0)=\eta_0$ in \eqref{eq_3_3} holds, with $ \xi_0=(\tau_0, \eta_0)\in \R\times\R^{n-1}$.  Note that $p$, $\varphi$ and $\psi$ are holomorphic in their variables. For a holomorphic function $f(z_1,\dots, z_N)$ in an open set $V\subset \C^N$ we write $f'_z(z)=(\p_{z_1}f(z), \dots, \p_{z_N}f(z))$ for the complex gradient, $f''_{zz}(z)=(\p_{z_jz_k}f(z))_{j,k=1}^N$ for the complex Hessian, etc. If $z_j=x_j+iy_j$, holomorphicity implies that $\p_z^\alpha f(z)=\p_x^\alpha f(z)$ for any multi-index $\alpha$. This shows that a holomorphic solution $\varphi$ of \eqref{eq_3_3} in $\tilde U\subset \C^n$ indeed yields a real analytic solution of \eqref{eq_3_1} in $U$.

\textbf{Remark.} Let us note that \cite[Theorem 5.5]{Grigis_Sjoistrand_book} gives the standard Hamilton-Jacobi theory locally near a point in the smooth case, and the extension of this theory to the holomorphic case is discussed in the remark following Theorem 1.8.2 in \cite{Hormander_book_old}. However here we need to construct the phase $\varphi$ enjoying the good properties along the entire geodesic segment, and therefore, we shall give a detailed discussion of the construction below. The condition \eqref{eq_3_4_1} will  be crucial for this purpose.

\textbf{Step 1.  Solving near a point. }
In order to solve \eqref{eq_3_3}, we start by following the proof of  \cite[Theorem 5.5]{Grigis_Sjoistrand_book}, see also \cite{Sjostrand_1975}. The setting of our proof is illustrated in Figure~\ref{pic:Flowout_of_Lambda_prime} below.  
 \begin{figure}[ht!]\label{pic:Flowout_of_Lambda_prime}
 \centering
  \includegraphics[scale=1.6]{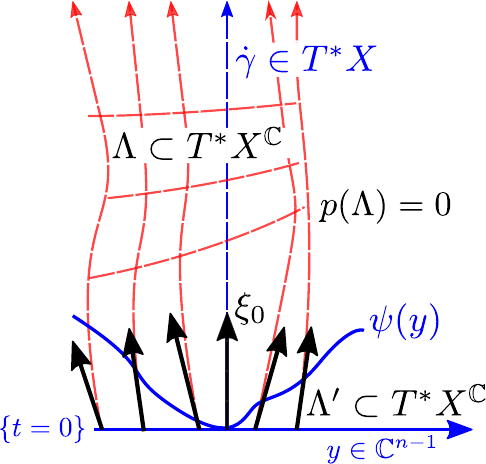}
 \caption{A Lagrangian submanifold $\Lambda \subset T^*X^\C$ satisfying $p\s(\Lambda)=0$ is the union of the red integral curves of $H_p$ in $T^*X^\C$ passing through $\Lambda'\subset T^*X^\C$, which is represented by the black arrows. Here $T^*X^\C$ is the cotangent bundle of the complexification of $X$, which is locally $\C^n_x\times \C^n_\xi$.}
 \label{pic:Flowout_of_Lambda_prime}
 \end{figure}
To this end, we observe first that in view of \eqref{eq_3_2_new_1}, by the implicit function theorem applied to $p(0,\ccdot,\ccdot,\ccdot)$, in a complex neighborhood of $(0,0,\tau_0, \eta_0)$ we have $p(0,y,\tau, \eta)=0$ if and only if $\tau=\lambda(y,\eta)$ where $\lambda$ is a holomorphic function near $(0, \eta_0)\in \C^{2(n-1)}$ such that $\lambda(0,\eta_0)=\tau_0$. 

Let us define
\[
\Lambda':=\{(0, y, \tau, \eta): \eta=\psi'_y(y), \tau=\lambda(y,\eta), y\in \text{neigh}(0,\C^{n-1}) \}\subset \C^{2n}.
\]
We have that $\Lambda'$ is a complex manifold of complex dimension $n-1$ such that 
\[
\Lambda'\subset p^{-1}(0),
\]
which is isotropic in the sense that the restriction of $\sigma$ to $T\Lambda'\times T\Lambda'$ vanishes:
\begin{equation}
\label{eq_new_sigma}
\sigma|_{\Lambda'}=0.
\end{equation}
  Here $\sigma=\sum_{j=1}^n d\xi_j\wedge dx_j$ is the complex symplectic form on $\C^{2n}=\C^n_x\times \C^n_\xi$. Indeed, any vector tangent to $\Lambda'$ is of the form $(0,V^y,V^{\tau},V^{\eta})$ with $V^\eta=\psi_{yy}''\s V^y$ and $V^y\in \C^{n-1}$. Applying $\sigma$ to two such vectors gives $V_1^y\cdot \psi_{yy}''\s V_2^y-V_2^y\cdot \psi_{yy}''\s V_1^y=0$, showing \eqref{eq_new_sigma}.

  %Thus if $Z_1$ and $Z_2$ are two  such vectors, then $\sigma(Z_1,Z_2)=V_1^y\cdot \psi_{yy}''\s V_2^y-V_2^y\cdot \psi_{yy}''\s V_1^y=0$.)} % with $V^y_1,V^y_2\in \C^{n-1}$)
%=(0,V_1^y,V_1^{\tau},V_1^{\eta})$ and $Z_2=(0,V_1^y,V_1^{\tau},V_1^{\eta})$ are tangent to $\Lambda'$,  it follows that $\sigma(T_1,T_2)=V_1^y\cdot \psi_{yy}''\cdot V_2^y-V_2^y\cdot \psi_{yy}''\cdot V_1^y=0$.)
%Any vector tangent to $\Lambda'$ is of the form $(0,V^y,V^{\tau},V^{\eta})$
%Let $(0,\overline{y},\overline{\tau},\overline{\eta})\in \Lambda'$. Then $X\in T_{(0,\overline{y},\overline{\tau},\overline{\eta})}\Lambda'$ if and only if $X=(0, V^y, \lambda_y'(\overline{y},\overline{\eta})\cdot V^y+\lambda_\eta'(\overline{y},\overline{\eta})\cdot \psi_{yy}''(\overline{y})\cdot V^y), \psi_{yy}''(\overline{y})$). We write $X=(V^x,V^\xi)$, where $V^x=(0,V^y)$. Then $\sigma(X,X)=\sum_{j=1}^nV_j^\xi V_j^x-V_j^xV_j^\xi$. Since the first component of $V^x$ is zero, we have that first component of $V^\xi$ is irrelevant. Consequently, we have $\sigma(X,X)=\sum_{j=2}^nV_j^\xi V_j^y-V_j^yV_j^\xi= V_y\cdot \psi_{yy}''(\overline{y})\cdot V^y-V^y\cdot \psi_{yy}''(\overline{y})\cdot V^y=0.$ Do you agree? If, I could write a hint for the reader how to see $\sigma|_{\Lambda'}=0$.
Note also that 
\begin{equation}
\label{eq_3_4_2_00-new}
\Lambda'\cap \R^{2n}=\{(0,0,\tau_0,\eta_0)\}. 
\end{equation}
 Indeed, $(0,0,\tau_0,\eta_0)\in \Lambda'\cap \R^{2n}$ as $(\tau_0,\eta_0)\in \R^n$ and $\psi'_y(0)=\eta_0$. To see the opposite inclusion, let $(0, y, \lambda(y,\eta), \eta=\psi'(y))\in  \Lambda'\cap \R^{2n}$ and Taylor expand $\psi'(y)$ at $y=0$, 
  \[
 \eta=\psi'(y)=\eta_0+\psi''(0)y+\mathcal{O}(|y|^2), \quad y\in \R^{n-1}.
 \]
 We have $\text{Im}\, \eta=\text{Im}\, \psi''(0)y+\mathcal{O}(|y|^2)$, and  therefore, in view of \eqref{eq_3_4_1},  $\text{Im}\, \eta =0$ implies that $y=0$. This shows \eqref{eq_3_4_2_00-new}.

Let $H_p$ be the complex Hamilton vector field of $p$, and let us consider the $H_p$ flowout of $\Lambda'$:  
\[
\Lambda=\bigg\{\exp\bigg(\frac{t}{2}H_p\bigg)(\rho): \rho\in \Lambda', t\in \text{neigh}(I,\C)\bigg\}\subset \C^{2n}. 
\]
Here if $\mu=\sum_{j=1}^N a_j(z)\p_{z_j}$ is a holomorphic vector field on an open set $V\subset \C^N$ in the sense that $a_j\in \text{Hol}(V)$, $j=1,\dots,n$, we can define the flow $\exp(t\mu)(\rho)$, $\rho\in V$, locally for $t\in \text{neigh}(0,\C)$, by solving the system of ODE, 
\[
\begin{cases}\dot{z}_j(t)=a_j(z(t)), \quad 1\le j\le n,\\
z(0)=\rho,
\end{cases}
\]
see \cite[Section 1]{Forstneric} and the references given there. 

Then $\Lambda'\subset \Lambda$, and since the flow of $H_p$ preserves $p$, we have 
\[
\Lambda\subset p^{-1}(0),
\]
and $\Lambda$ is a $\C$--Lagrangian submanifold of $\C^{2n}$, see \cite[Proposition 5.4]{Grigis_Sjoistrand_book} for a proof in the real case. The proof in the present holomorphic setting is similar.  Let us also recall from \cite[page 60]{Grigis_Sjoistrand_book} that the holomorphic Hamilton vector field  $H_p$ is tangent to $\Lambda$ at each point of $\Lambda$. This is because $\Lambda$ is a Lagrangian contained in $p^{-1}(0)$.

The differential of $\pi_x|_{\Lambda}$ is bijective at $(0,0,\tau_0,\eta_0)$ since the differential of $\pi_x$ is injective  and since any Lagrangian submanifold has dimension $\dim(X)$. (The differential of $\pi_x|_\Lambda$ is injective since the differential of the exponential map $TX\to X$ is injective.)
Consequently, there is a function $\varphi\in \text{Hol}(\text{neigh}(0,\C^n))$ such that
\begin{equation}
\label{eq_3_4_2}
\Lambda=\Lambda_\varphi:=\{(x,\varphi'_x(x)): x\in \text{neigh}(0,\C^n)\},
\end{equation}
see \cite[Section 5.6, Exercise 4]{Martinez_book}, and also \cite[Theorem 5.3]{Grigis_Sjoistrand_book} for the real version of this result. We have $\varphi'_x(0)=\xi_0$ and modifying $\varphi$ by a constant we get $\varphi(0,y)=\psi(y)$, and such a solution is unique.

\textbf{Step 2. Solving near $\gamma$.}
Let us denote the tangent space of $\Lambda$ at $(0,0,\tau_0,\eta_0)$ by $\Lambda_0$ and write
\begin{equation}
\label{eq_3_7_1}
\Lambda_0:=T_{(0,0,\tau_0, \eta_0)}\Lambda=\{(\delta_x,\delta_\xi)\in \C^n\times \C^n: \delta_\xi=\varphi''_{xx}(0)\delta_x\},
\end{equation}
where in the second equality we used  \eqref{eq_3_4_2}. 

We claim that $\Lambda_0$ is a positive Lagrangian plane in the sense that 
\[
\frac{1}{i}\sigma(\rho,\overline{\rho})\ge 0,\quad \rho\in \Lambda_0.
\]
To this end, letting $M_0=\varphi''_{xx}(0)$ and using~\eqref{eq_3_7_1}, we write   
 $\rho=(\delta_x,M_0\delta_x)\in \Lambda_0$. Then using that $M_0$ is symmetric, we get
\begin{equation}
\label{eq_3_13_1}
\begin{aligned}
\frac{1}{i}\sigma(\rho,\overline{\rho})=\frac{1}{i}(M_0\delta_x\cdot \overline{\delta_x}-\overline{M_0\delta_x}\cdot\delta_x)=
2\text{Im}\, (M_0\delta_x\cdot\overline{\delta_x})\\
=2\text{Im}\, (M_0) \text{Re}\,\delta_x\cdot \text{Re}\, \delta_x+ 2\text{Im}\, (M_0)\text{Im}\delta_x\cdot \text{Im}\delta_x,
\end{aligned}
\end{equation}
and therefore, it suffices to prove that 
\begin{equation}
\label{eq_3_13}
\text{Im}\, M_0\ge 0.
\end{equation}  
In doing so, using \eqref{eq_3_3}, we write 
\begin{equation}
\label{eq_3_11}
M_0=\begin{pmatrix}  \varphi''_{tt}(0,0) & \varphi''_{ty}(0,0)\\
\varphi''_{yt}(0,0) & \psi''_{yy}(0)
\end{pmatrix}.
\end{equation}
Using that $H_p$ is tangent to $\Lambda_\varphi$, we see that $\exp(\frac{t}{2}H_p)(0,\xi_0)=(x(t),\varphi'_x(x(t)))$
 is real for $t\in \text{neigh}(0, \R)$, so that $\varphi'_t(t,0)$, $\varphi'_y(t,0)$ are real. Hence, 
\begin{equation}
\label{eq_im_M_0}
\text{Im}\,M_0=\begin{pmatrix}  0 & 0\\
0 & \text{Im}\, \psi''_{yy}(0)
\end{pmatrix},
\end{equation}
and therefore, by the condition $\text{Im}\, \psi''_{yy}(0)>0$ we imposed on $\psi$ in~\eqref{eq_3_4_1}, \eqref{eq_3_13} follows.

For future reference, let us remark that 
\begin{equation}
\label{eq_im_M_0_1}
\Lambda_0\cap \R^{2n}=\R \s  H_p(0,\xi_0),
\end{equation}
where $\R \s H_p(0,\xi_0)=\{s\s H_p(0,\xi_0): s\in \R\}$. 
Indeed, we have $H_p(0,\xi_0)\in \Lambda_0\cap \R^{2n}$ since the $H_p$ vector field is tangent to $\Lambda$. On the other hand, if $(\delta_x, M_0\delta_x)\in \Lambda_0\cap \R^{2n}$, it follows from \eqref{eq_im_M_0} that 
\[
\delta_x=(\delta_t, 0)=\delta_t\dot{x}(0)=\delta_t p'_{\xi}(0,\xi_0), 
\]
where $\delta_t\in \R$. Here in the second equality we used that $(t,0)$ corresponds to the geodesic in Fermi coordinates. We get $(\delta_x,M_0\delta_x)=\delta_t(p'_\xi(0,\xi_0),M_0p'_\xi(0,\xi_0))=\delta_t H_p((0,\xi_0))$, which shows \eqref{eq_im_M_0_1}. Here in the last equality we used $H_p(0,\xi_0)\in \Lambda_0$.

Let
\[
\kappa(t):=\exp\bigg(\frac{t}{2} H_p\bigg): \Lambda \to \Lambda, \quad t\in I\subset \R, 
\]
and therefore, the differential satisfies 
\[
d\kappa(t)(0,\xi_0):\Lambda_0\to T_{\kappa(t)(0,\xi_0)}\Lambda.
\]
As the canonical transformation $\kappa(t)$ is real for each $t\in I$, $d\kappa(t)(0,\xi_0)$ preserves positivity, see  \cite[Section 5.6, Exercise 8]{Martinez_book},  and therefore, 
\[
\Lambda_t:=T_{\kappa(t)(0,\xi_0)}\Lambda\subset \C^{2n} 
\]
is a positive Lagrangian plane, for all $t\in I$.

We claim that $\Lambda_t$ is transversal to the fiber  $F=\{(0,\eta):\eta\in \C^n\}\subset \C^{2n}$, for all $t\in I$, i.e. $\Lambda_t+F=\C^{2n}$. 
 As $\dim \Lambda_t=n$, we have to show that $\Lambda_t\cap F=\{0\}$. Indeed, let $(0,\eta)\in \Lambda_t\cap F$. Then 
\eqref{eq_3_7_1} implies that 
\begin{equation}
\label{eq_3_18}
\begin{pmatrix}0\\ 
\eta
\end{pmatrix}=d\kappa (t)(0,\xi_0) \begin{pmatrix}
\delta_x\\
M_0\delta_x
\end{pmatrix}, 
\end{equation}
for some $\delta_x\in \C^n$.
We have 
\begin{align*}
0=\frac{1}{i}\sigma\bigg( \begin{pmatrix}0\\ 
\eta
\end{pmatrix}, 
\overline{\begin{pmatrix}0\\ 
\eta
\end{pmatrix}} \bigg)= \frac{1}{i}\sigma\bigg(d\kappa (t)(0,\xi_0) \begin{pmatrix}
\delta_x\\
M_0\delta_x
\end{pmatrix}, 
d\kappa (t)(0,\xi_0) \overline{\begin{pmatrix}
\delta_x\\
M_0\delta_x
\end{pmatrix}} \bigg)\\
= \frac{1}{i}\sigma\bigg( \begin{pmatrix}
\delta_x\\
M_0\delta_x
\end{pmatrix}, 
 \overline{\begin{pmatrix}
\delta_x\\
M_0\delta_x
\end{pmatrix}} \bigg)=2\text{Im}\, ( M_0 \delta_x\cdot\overline{\delta_x}).
\end{align*}
As $\text{Im}\, M_0\ge 0$, we get $(\text{Im}\, M_0) \delta_x=0$, and therefore,  \eqref{eq_im_M_0} implies that 
$\delta_x=\alpha p'_\xi(0,\xi_0)$ for some $\alpha\in \C$.  Thus, by \eqref{eq_3_18}  we obtain that 
\[
\begin{pmatrix} 0\\
\eta
\end{pmatrix} = d\kappa (t)(0,\xi_0) (\alpha H_p(0,\xi_0))=\alpha H_p(x(t),\xi(t))=\alpha\begin{pmatrix} \dot{x}(t)\\
\dot{\xi}(t)
\end{pmatrix}.
\]
Since $\dot{x}(t)\ne 0$, we get $\alpha=0$. Hence, 
\[
\eta=0,
\] 
which establishes the claim.

As $\Lambda_t$ is transversal to the fiber for all $t\in I$, by inspection of the proof of Theorem 5.5 in  \cite{Grigis_Sjoistrand_book}, we conclude that there exists $\varphi \in \text{Hol}(\text{neigh}(I\times B,\C^n))$ such that $\Lambda=\Lambda_\varphi$ and $\varphi$ solves \eqref{eq_3_3}.  The function $\varphi$ is a continuation of the one appearing in~\eqref{eq_3_4_2}. Notice that it is precisely thanks to the fact that the tangent plane $\Lambda_t$ does not contain any non-zero vector of the form $(0,\eta)$ for all $t\in I$ that the proof of Theorem 5.5 in  \cite{Grigis_Sjoistrand_book} applies near each point in $I\times \{0\}$, see also \cite[Section 24.2]{Hormander_bookI}.

\textbf{Step 3. Properties of the solution.}
Next we shall check that the property \eqref{eq_3_4}, that is
\begin{equation*}
%\label{eq_3_4}
\text{Im}\,  \varphi (t,y)\ge 0, \quad \text{Im}\, \varphi(t,0) =0, \quad  \text{Im}\, \varphi''_{yy}(t,0)>0,\quad t\in I, 
\end{equation*}
holds for $\varphi$.  First, $\varphi'_x(x(t))=\xi(t)$ is real for $t\in I$. 
Writing 
\[
\frac{d}{dt}\varphi(x(t))=\varphi'_x(x(t))\cdot\dot{x}(t)=\xi(t)\cdot \frac{1}{2}p'_\xi(x(t),\xi(t)),
\]
we have 
\begin{equation}
\label{eq_3_17_new_0}
\varphi(t,0)=\psi(0)+\frac{1}{2}\int_0^t \xi(s)\cdot p'_\xi(x(s),\xi(s))ds=\psi(0)+t, 
\end{equation}
as $\xi\ccdot p'_\xi(x,\xi)=2(p(x,\xi)+1)$. 
Thus, using that $\psi(0)$ is real, we see that $\text{Im}\, \varphi(t,0)=0$ for $t\in I$. Furthermore, if $\psi(0)=0$, we get $\varphi(t,0)=t$.

Let $M(t)=\varphi''_{xx}(x(t))$. Then $M(t)$ is an $n\times n$ complex symmetric matrix depending real analytically on $t$, such that 
\begin{equation}
\label{eq_3_17}
\text{Im}\, M(t)\ge 0,
\end{equation} 
 in view of the positivity of $\Lambda_t$. We claim that 
\begin{equation}
\label{eq_3_16}
\text{Im}\, M(t)|_{W}>0,
\end{equation}
where $W\subset \R^{n}$ is an algebraic supplement to $\R\s \dot{x}(t)$ so that $\R\s \dot{x}(t)\oplus W=\R^n$. 
To that end, let us observe first that 
\[
\Lambda_t\cap\R^{2n}=d\kappa(t)(0,\xi_0)(\Lambda_0\cap \R^{2n})=d\kappa(t)(0,\xi_0) (\R \s H_p(0,\xi_0))=\R\s H_p(x(t), \xi(t)).
\]
Here we have used \eqref{eq_im_M_0_1} in the second equality. Let $v\in W$ be such that 
\[
\text{Im}\, M(t) v\cdot v=0.
\]
Hence, by \eqref{eq_3_17}, we get 
\[
\text{Im}\, M(t) v=0.
\]
Thus, $(v, M(t)v)\in \Lambda_t\cap\R^{2n}=\R \s H_p(x(t), \xi(t))$, and therefore, $v$ is proportional to $p'_\xi(x(t), \xi(t))=\dot{x}(t)$. This gives that $v=0$, since $v\in W$. Hence, \eqref{eq_3_16} follows, and we get $\text{Im}\, \varphi''_{yy}(t,0)>0$ for  all $t\in I$. 

Finally, we get $\text{Im}\,\varphi(t,y)\ge 0$ for all $(t,y)\in U$ by Taylor's formula and by using that $\varphi'_x(x(t))=\xi(t)$ is real.  We have therefore constructed a real analytic solution $\varphi$ of \eqref{eq_3_1} such that  \eqref{eq_3_4} holds.

\subsection{Construction of the amplitude}

We shall follow \cite[Theorem 9.3]{Sjostrand_Analytic}, where the construction of the amplitude as a classical analytic symbol is carried out in a neighborhood of a point, extending the construction to a full neighborhood of a geodesic segment. 

We look for the amplitude $a$ in the form of a formal power series in $h$, 
\begin{equation}
\label{eq_new_4_0}
a(x;h)=\sum_{k=0}^\infty h^k a_k(x).
\end{equation}
%\tbl{so that $a$ is a classical analytic symbol. The function $a$ is classical analytic symbol if for any compact set $K$, there is $N$ and $C$ (independent of $N$) such that
%\[
% \sup_{x\in K}\abs{a(x,h)-\sum_{k=0}^{N-1}h^ka_k(x)}\leq C^{N+1}N^N,
%\]
%and if for all integers $k\geq 0$ the functions $a_k$ are holomorphic and 
%\[
% \sup_{x\in K}\abs{a_k(x)}\leq C^kk^k.
%\]
%~\f{I couldn't immediately find a standard reference for this, but used https://arxiv.org/pdf/1808.00199.pdf}
%}
From 
\eqref{eq_3_0},  we see that we want to solve the following equation formally in powers of $h$,
\begin{equation}
\label{eq_new_4_1}
e^{-is\varphi}(-h^2\Delta_g-(hs)^2)e^{is\varphi}a=[-hiL_0 -ih \Delta_g\varphi +h^2(-\Delta_g+\lambda L_0+\lambda\Delta_g\varphi) ]a=0,
\end{equation}
in a fixed complex domain $\tilde U$, containing $\Gamma$. Here 
\begin{equation}
\label{eq_new_4_1_0}
L_0=2\langle d\varphi, d\ccdot \rangle_g=2G(x)\varphi'_x\cdot \p_x=p'_\xi(x,\varphi'_x(x))\cdot \p_x,
\end{equation}
where $p$ is given in \eqref{eq_3_1_2}.
The transport equation \eqref{eq_new_4_1} can be written in the following form, 
\begin{equation}
\label{eq_new_4_1_1}
(hL_0+hf(x)+h^2Q(x,D_x))a=0,
\end{equation}
where $f(x)=\Delta_g\varphi$ is a holomorphic function on $\tilde U$ and $Q(x,D_x)=i(-\Delta_g+\lambda L_0+\lambda\Delta_g\varphi)$ is a holomorphic differential operator of order 2.  
To solve \eqref{eq_new_4_1_1},  we remark first that the holomorphic vector field $L_0$ is transversal to each complex hypersurface $H_{t_0}=\{(t, y)\in \text{neigh}(I,\C)\times\text{neigh}(0,\C^{n-1}): t=t_0\in I\}$ at $(t_0,0)$. Indeed, 
\[
p'_\xi(x(t),\varphi'_x(x(t)))\cdot \p_x=p'_\tau(x(t),\varphi'_x(x(t)))\p_t+p'_\eta(x(t),\varphi'_x(x(t)))\cdot\p_y, 
\]
where $p'_\tau(x(t),\varphi'_x(x(t)))\ne 0$ for all $t\in I$ since $\p_\tau p(x(t),\xi(t))\ne 0$ for all $t \in I$ as noted in \eqref{eq_3_2_new_1}.  Thus, substituting \eqref{eq_new_4_0} into  \eqref{eq_new_4_1_1}, we get a sequence of transport equations which can all be solved uniquely in a suitable complex domain containing $\Gamma$, provided that $a|_{H_{t_0}}$ is prescribed, for some $t_0\in I$.  However, the difficulty here is that we would like our solution $a(x;h)$ to be a classical analytic symbol, and following \cite[Section 9]{Sjostrand_Analytic}, we shall establish this fact making use of the method of "nested neighborhoods" introduced in \cite{Sjostrand_Analytic}. Contrary to  \cite[Theorem 9.3]{Sjostrand_Analytic}, where the family of "nested neighborhoods" is considered near a point, here we shall work in such neighborhoods near a piece of the geodesic.    

For simplicity, let us take $t_0=0$. We look for solution to~\eqref{eq_new_4_1_1} by using convenient coordinates. The coordinates we will use are the usual flowout coordinates (see e.g.~\cite{lee_book}), which we show to exist for $L_0$ on a neighborhood of a given interval.
\begin{lem}
\label{lem_holom_coor}
Let $J\subset\subset I$ be an open interval. There exist local holomorphic coordinates $(s,z)\in \text{neigh}(\overline{J},\C)\times \text{neigh}(0,\C^{n-1})$ such that the hyperplane $H_0$ is given by the equation $s=0$ and  $L_0=\frac{\p}{\p s}$.
\end{lem}
\begin{proof}
We continue to work in the Fermi coordinates $x=(t,y)$ and recall from  \cite{Kenig_Salo_APDE_2013} that 
\begin{equation}
\label{eq_new_4_1_1_1}
G(t,y)=(g^{jk}(t,y))=1+ \mathcal{O}(|y|^2).
\end{equation}
Now \eqref{eq_3_1}, \eqref{eq_3_1_2}, and \eqref{eq_new_4_1_1_1} imply that 
\begin{equation}
\label{eq_new_4_1_1_3}
(\varphi'_t)^2(t,0)+(\varphi'_y)^2(t,0)=1, 
\end{equation}
and therefore, it follows from \eqref{eq_3_17_new_0}  and \eqref{eq_new_4_1_1_3} that $\varphi'_y(t,0)=0$. 
Hence, Taylor expanding $\varphi(t,y)$ at $y=0$, we get 
\begin{equation}
\label{eq_new_4_1_1_4}
\varphi(t,y)=\psi(0)+t+\mathcal{O}(|y|^2).
\end{equation}
It follows from \eqref{eq_new_4_1_0}, \eqref{eq_new_4_1_1_1}, and \eqref{eq_new_4_1_1_4} that 
\begin{equation}
\label{eq_new_4_1_1_5}
L_0=2(1+\mathcal{O}(|y|^2))\begin{pmatrix} 1+\mathcal{O}(|y|^2)\\
\mathcal{O}(|y|) 
\end{pmatrix}\cdot \begin{pmatrix}\p_t\\
\p_y
\end{pmatrix}=2(1+\mathcal{O}(|y|^2))\p_t+\mathcal{O}(|y|)\cdot \p_y.
\end{equation}
Consider the initial value problem for the  flow $\exp(sL_0)(0,z)$, 
\begin{equation}
\label{eq_new_4_1_1_6}
\begin{cases} 
\p_s(t,y)(s,z)=L_0((t,y)(s,z)),\\
(t,y)(0,z)=(0,z),
\end{cases}
\end{equation}
where $(s,z)\in \text{neigh}(I,\C)\times \text{neigh}(0,\C^{n-1})$ .
In particular, $y(s,z)|_{z=0}=0$ and therefore, $y(s,z)=\mathcal{O}(|z|)$. Differentiating the first equation in \eqref{eq_new_4_1_1_6} in $z_j$ and using \eqref{eq_new_4_1_1_5}, we get 
\begin{equation}
\label{eq_new_4_1_1_7}
\begin{cases} 
\p_s(\p_{z_j} t(s,z))= \mathcal{O}(y(s,z)\p_{z_j}y)=\mathcal{O}(|z|),\\
\p_{z_j}t(0,z)=0. 
\end{cases}
\end{equation}
Hence, 
\begin{equation}
\label{eq_new_4_1_1_8}
\p_{z_j}t(s,z)=\mathcal{O}(|z|).
\end{equation}
Consider the holomorphic map
\[
F:  \text{neigh}(I,\C)\times \text{neigh}(0,\C^{n-1})\ni (s,z)\mapsto (t,y)(s,z).
\]
In view of \eqref{eq_new_4_1_1_8}, the differential $DF(s,0)$ is given by  
\begin{equation}
\label{eq_new_4_1_1_9}
DF(s,0)=\begin{pmatrix} t'_s(s,0) & 0& \dots &0\\
y'_{1s}(s,0)&  y'_{1z_1} (s,0)& \dots &y'_{1z_{n-1}}(s,0)\\
\vdots &  \vdots & & \vdots \\
y'_{n-1 s} (s,0)& y'_{n-1 z_1}(s,0)& \dots & y'_{n-1 z_{n-1}}(s,0)
\end{pmatrix},
\end{equation}
where $t'_s(s,0)=2(1+\mathcal{O}(|y(s,0)|^2))=2$. By Liouville's formula, see \cite[Theorem 1.2.5]{Hormander_book_hyperbolic}, we know that the last $n-1$ columns in \eqref{eq_new_4_1_1_9} are linearly independent, and therefore, $\det(DF(s,0))\ne 0$ for all $ s\in\text{neigh}(I,\C)$.  Furthermore, $F|_{I\times \{0\}}$ is injective as $F(s,0)=(t(s,0),0)=(2s,0)$. An application of a holomorphic version of \cite[Lemma 7.3]{Kenig_Salo_APDE_2013} allows us to conclude that  $F$ is a holomorphic diffeomorphism in $\text{neigh}(\overline{J},\C)\times \text{neigh}(0,\C^{n-1})$ where $J\subset\subset I$ is an open interval. 

Now writing $x=(t,y)$, in view of \eqref{eq_new_4_1_1_6}, we see that 
\[
\frac{\p}{\p s} u(x(s,y))=u'_x (x(s,y))\cdot \dot{x}(s,y)=(L_0u)(x(s,y)). 
\]

Finally, it follows from \eqref{eq_new_4_1_1_5} and \eqref{eq_new_4_1_1_6} that 
\[
\begin{cases} \p_s t(s,z)=2(1+\mathcal{O}(|z|^2)),\\
t(0,z)=0,
\end{cases}
\]
and therefore, $t(s,z)=2s+\mathcal{O}(|z|^2)s$. Hence, $t=0$ is equivalent to the fact that $s=0$, showing that the hyperplane $H_0$ is given by the equation $s=0$.
\end{proof}

Passing to the new holomorphic coordinates provided by Lemma \ref{lem_holom_coor}, and renaming them as $x=(t,y)$, we are led from \eqref{eq_new_4_1_1} to consider the following initial value problem, 
\begin{equation}
\label{eq_new_4_2}
\begin{cases}
\big(h\frac{\p}{\p t}+hf(x)+h^2Q(x,D_x)\big)a=0,\\
a|_{t=0}=w(y;h),
\end{cases}
\end{equation}
where $w(y;h)$  is a classical analytic symbol near $0\in \C^{n-1}$. We would like to find a classical analytic symbol $a$ solving  \eqref{eq_new_4_2}. Here $f$ is a holomorphic function, and $Q$ is a holomorphic differential operator of order $2$. To that end, it suffices  to solve the following problem,  
\begin{equation}
\label{eq_new_4_3}
\begin{cases}
\big(h\frac{\p}{\p t}+hf(x)+h^2Q(x,D_x)\big)a=hv,\\
a|_{t=0}=0,
\end{cases}
\end{equation}
where $v(x;h)$ is a classical analytic symbol in $\text{neigh}(\overline{J},\C)\times \text{neigh}(0,\C^{n-1})$. This is because a solution $a$ to~\eqref{eq_new_4_3} with $v=-(\frac{\p}{\p t}+f(x)+hQ(x,D_x))v_0$ and $v_0|_{t=0}=w$, implies that $a+v_0$ solves~\eqref{eq_new_4_2}.  Using that 
\[
\p_t+f(t,y)=e^{-F(t,y)}\circ \p_t \circ e^{F(t,y)},
\]
where $F'_t(t,y)=f(t,y)$, we may assume that $f(x)=0$.  

We shall first carry out the analysis of \eqref{eq_new_4_3} under the assumption that the interval $J$ is symmetric about the origin and after a rescaling we may assume that $\overline{J}=[-1,1]$.  Let $\Omega\subset \C^n$ be open such that 
$[-1,1]_t\times \{0\}_{y}\subset \Omega$ and $\Omega$ is in the domain of definition of various symbols.  Then let $0<\varepsilon<1$, $r>0$ be small but fixed so that  if we set 
\[
\Omega_0=\big\{(t,y)\in \C^n: \frac{|y|}{\varepsilon}+\frac{|\text{Im}\, t|}{\varepsilon}+|\text{Re}\, t|<1+r\big\}
\]
then $\overline{\Omega_0}\subset \Omega$. Consider the family of open sets,
\[
\Omega_s=\big\{(t,y)\in \C^n: \frac{|y|}{\varepsilon}+\frac{|\text{Im}\, t|}{\varepsilon}+|\text{Re}\, t|<1+r-s\big\},
\]
with $0\le s< r$.  Note that $\Omega_s$ is a family of "nested neighborhoods" of $[-1,1]\times 0$ in the sense of  \cite[Theorem 9.3]{Sjostrand_Analytic}, so that we have
\begin{itemize}
\item[(i)] if $s_1>s_2$ then $\Omega_{s_1}\subset \Omega_{s_2}$,
\item[(ii)] there exists $\delta>0$ such that for all $s_1>s_2$ and all $x\in \Omega_{s_1}$ we have the inclusion $B_{\C^n}(x,\delta(s_1-s_2))\subset \Omega_{s_2}$. 
\end{itemize}

Given $\mu>0$, we say that $a\in\mathcal{A}_{\mu}$, if $a(x;h)=\sum_{k=0}^\infty a_k(x)h^k$, $a$ is holomorphic in $\Omega$, such that for all $s\in (0,r)$,
\begin{equation}
\label{eq_new_4_4}
\sup_{\Omega_s} |a_k|\le \frac{f(a,k)}{s^k}k^{k},
\end{equation}
where $f(a,k)$ is the best constant for which \eqref{eq_new_4_4} holds, and 
\begin{equation}
\label{eq_new_4_4_1}
\sum_{k=0}^\infty f(a,k)\mu^k:=\|a\|_{\mu}<\infty. 
\end{equation}
Now if $a\in \mathcal{A}_{\mu}$ for some $\mu>0$ then $f(a,k)\le C^{k+1}$, $k=0, 1, 2,\dots$, and therefore, $a$ is a classical analytic symbol on $\Omega_0$. Let 
\begin{equation}
\label{eq_new_4_5}
(\p_t^{-1} a)(t,y)=\int_0^t a(\tau, y)d\tau.  
\end{equation}
We shall need the following result, see \cite[Theorem 9.3]{Sjostrand_Analytic} and \cite[Lemma 5.5]{Robbiano_Zuily}. 

\begin{lem}
\label{lem_analytic_symbol_1}
Let $a\in \mathcal{A}_\mu$ be of the form 
\[
a=\sum_{k=2}^\infty h^ka_k,
\]
and let $b=(h\p_t)^{-1}a$. Then  
\begin{equation}
\label{eq_new_4_5_1}
\|b\|_\mu\le \mathcal{O}\bigg(\frac{1}{\mu}\bigg)\|a\|_\mu. 
\end{equation}
\end{lem}
\begin{proof}
We have
\[
b=\sum_{k=2}^\infty h^{k-1}\p_t^{-1}a_k=\sum_{k=1}^\infty h^k b_k,
\]
where $b_k=\p_t^{-1}a_{k+1}$. Let us estimate $\sup_{\Omega_s}|b_k|$. To that end, we write 
\[
b_k(x)=t\int_0^1 a_{k+1}(\sigma t, y)d\sigma
\]
We claim that for $0\le \sigma\le 1$,  if $x=(t,y)\in \Omega_s$ then 
\begin{equation}
\label{eq_new_4_6}
(\sigma t,y)\in \Omega_{s+(1-\sigma)|t|}.
\end{equation} 
Indeed, using that $0<\varepsilon<1$, we get 
\begin{align*}
\frac{|y|}{\varepsilon}+\frac{\sigma |\text{Im}\, t|}{\varepsilon}+\sigma|\text{Re}\, t|&< 1+r-s -\frac{(1-\sigma)}{\varepsilon}|\text{Im}\, t|-(1-\sigma)|\text{Re}\, t|\\
&<1+r-(s+(1-\sigma)|t|),
\end{align*}
showing $(\sigma t,y)\in \Omega_{s+(1-\sigma)|t|}$ as claimed. It follows from \eqref{eq_new_4_4}, \eqref{eq_new_4_6} that for $x\in \Omega_s$, we have
\begin{align*}
|b_k(x)|&\le |t| f(a,k+1)(k+1)^{k+1} \int_0^1\frac{d\sigma}{(s+(1-\sigma)|t|)^{k+1}} \\
&=f(a,k+1)(k+1)^{k+1}|t|  \int_0^1\frac{d\sigma}{(s+\sigma |t|)^{k+1}} \\
&=f(a,k+1)(k+1)^{k+1} \int_0^{|t|}\frac{d\sigma}{(s+\sigma )^{k+1}} \\
&\le
f(a,k+1) 
(k+1)^{k+1} \int_0^{\infty}\frac{d\sigma}{(s+\sigma )^{k+1}}=f(a,k+1)(k+1)^{k+1} \int_s^{\infty}\frac{d\sigma}{\sigma ^{k+1}}\\
&=f(a,k+1) \frac{(k+1)^{k+1}}{ks^k}. 
\end{align*}
Here we have used that $k\ge 1$. Thus, for any $0< s<r$, we get 
\[
\sup_{\Omega_s}|b_k|\le f(a,k+1)\frac{(1+1/k)k^k(1+1/k)^k}{s^k}\le \frac{2ef(a,k+1)}{s^k}k^k, 
\]
and therefore by the definition of $f(b,k)$, see~ \eqref{eq_new_4_4}, we have
\[
f(b,k)\le 2ef(a,k+1), \quad k=1,2,\dots. 
\]
Using \eqref{eq_new_4_4_1}, we obtain that 
\[
\|b\|_{\mu}=\sum_{k=1}^\infty f(b,k)\mu^k\le \sum_{k=1}^\infty 2ef(a,k+1)\mu^k=\frac{2e}{\mu}\|a\|_{\mu},
\]
establishing \eqref{eq_new_4_5_1}. 
\end{proof}

Now applying to $(h\p_t)^{-1}$ to  \eqref{eq_new_4_3}, we get 
\begin{equation}
\label{eq_new_4_6_1}
a+ (h\p_t)^{-1} h^2Q(x,D_x)a=\p_{t}^{-1}v.
\end{equation}
Here $\p_{t}^{-1}v$ is a classical analytic symbol in $\Omega_0$. 
To proceed, we need the following result. 
\begin{lem}
\label{lem_analytic_symbol_2}
Let $a\in \mathcal{A}_\mu$. Then $(h\p_t)^{-1} h^2Q(x,D_x)a\in \mathcal{A}_{\mu}$ with 
\begin{equation}
\label{eq_new_4_7}
\|(h\p_t)^{-1} h^2Q(x,D_x)a\|_{\mu}\le \mathcal{O}(\mu)\|a\|_{\mu}
\end{equation}
\end{lem}
\begin{proof}
Writing $a(x)=\sum_{k=0}^\infty a_k(x)h^k$, we get 
\[
h^2Q(x,D_x)a=\sum_{k=2}^\infty h^kQ(x,D_x)a_{k-2}.
\]
For $s_1>s_2$, in view of the property (ii) of the "nested neighborhoods" $\Omega_s$,  and \eqref{eq_new_4_4}, we obtain for $k=2,3,\dots$ that 
\begin{equation}
\label{eq_new_4_8}
\sup_{\Omega_{s_1}}|Q(x,D_x)a_{k-2}|\le \frac{C}{(s_1-s_2)^2}\sup_{\Omega_{s_2}}|a_{k-2}|\le  \frac{C}{(s_1-s_2)^2} \frac{f(a,k-2)}{s_2^{k-2}}(k-2)^{k-2}. 
\end{equation}
 The Cauchy estimate was used here in the first inequality. Taking $0< s_2=\frac{k-2}{k}s_1<s_1$ for $k=3,4,\dots$, we get from \eqref{eq_new_4_8} that 
\[
\sup_{\Omega_{s_1}}|Q(x,D_x)a_{k-2}|\le \frac{Cf(a,k-2)}{s_1^k}k^k,
\]
and therefore, in view of \eqref{eq_new_4_4}, 
\[
f(Q(x,D_x)a_{k-2},k)\le Cf(a,k-2). 
\]
Thus, by the definition of $\norm{\ccdot}_\mu$, see~\eqref{eq_new_4_4_1}, we obtain that 
\begin{equation}
\label{eq_new_4_9}
\|h^2Q(x,D_x)a\|_{\mu}\le \sum_{k=2}^\infty \mu^k Cf(a,k-2)\le \mathcal{O}(\mu^2)\|a\|_\mu. 
\end{equation}
Lemma \ref{lem_analytic_symbol_1} together with \eqref{eq_new_4_9} implies that 
\[
\|(h\p_t)^{-1} h^2Q(x,D_x)a\|_{\mu}\le \mathcal{O}(1/\mu)\|h^2Q(x,D_x)a\|_{\mu}\le 
\mathcal{O}(\mu)\|a\|_{\mu},
\]
establishing \eqref{eq_new_4_7}. 
\end{proof}
It follows from Lemma \ref{lem_analytic_symbol_2} that 
\[
\|\tilde Q^jw\|_{\mu} \le \mathcal{O}(\mu^j)\|w\|_\mu, \quad \tilde Q:=(h\p_t)^{-1} h^2Q(x,D_x),
\]
for $w\in \mathcal{A}_\mu$, and therefore, by Neumann series argument, we have that
the equation \eqref{eq_new_4_6_1},
\[
a+ (h\p_t)^{-1} h^2Q(x,D_x)a=\p_{t}^{-1}v,
\]
has a unique solution $a$ with $\|a\|_{\mu}<\infty$ for $\mu>0$ small enough. Thus, $a$ is  a classical analytic symbol in $\Omega_0$.  

We shall next proceed to solve \eqref{eq_new_4_2} when the interval $J$ is not necessarily symmetric with respect to the origin, $\overline{J}=[a,b]$ where  $a<0<b$. Without loss of generality, we may assume that $0<a+b$. Let $N\ge 0$ be the largest integer such that $(N+1)|a|<b$.  We first solve  \eqref{eq_new_4_2} with the initial condition prescribed at $t=0$ in the symmetric interval $[a, |a|]$ and obtain a unique classical analytic symbol $a^{(0)}$ in a complex neighborhood of $[a, |a|]\times \{0\}$. Next  we solve the initial value problem,
\begin{equation}
\label{eq_new_4_10}
\begin{cases}
(h\p_t +h^2Q(x,D_x))a^{(1)}=0,\\
a^{(1)}|_{t=|a|}=a^{(0)}|_{t=|a|},
\end{cases}
\end{equation}
in a complex neighborhood of  $[0, 2|a|]\times \{0\}$. Continuing this process and working the symmetric intervals of the form $[(j-1)|a|, (j+1)|a|]$, $j=2,\dots, N$,  we construct a classical analytic symbol in a complex neighborhood of $[a,(N+1)|a|]\times \{0\}$ solving \eqref{eq_new_4_2}.  Finally solving \eqref{eq_new_4_10} with the initial condition prescribed at $t=(N+1)|a|$ in a complex neighborhood of $[2(N+1)|a|-b, b]\times \{0\}$, we obtain  a classical analytic symbol in a complex neighborhood of $[a,b]\times \{0\}$ solving \eqref{eq_new_4_2}. 

Furthermore, demanding that $a|_{t=0}$ should be an elliptic classical analytic symbol near $0$ in $\C^{n-1}$, we conclude that the classical analytic symbol $a(x;h)$ is elliptic in the sense that $a_0\ne 0$. This completes  the construction of the amplitude as an elliptic classical analytic symbol.

It follows from  \eqref{eq_new_4_1} that for all $N\ge 1$, 
\[
e^{-is \varphi}(-h^2\Delta_g-(hs)^2)e^{is \varphi}\bigg( \sum_{j=0}^N h^j a_j\bigg)=h^{N+2}(-\Delta_g a_N+\lambda (L_0+\Delta_g\varphi)a_N)
\]
in a complex neighborhood $\tilde U$ of $\Gamma$. 
Using \eqref{eq_4_2} and Cauchy's estimates, we obtain after an arbitrarily small decrease of $\tilde U$ that 
\[
\bigg| e^{-is \varphi}(-h^2\Delta_g-(hs)^2)e^{is \varphi}\bigg( \sum_{k=0}^N h^k a_k\bigg)\bigg|\le h^{N+2}C^{N+1}N^N.
\]
Choosing $N=N(h)=[\frac{1}{heC}]$, we obtain  that 
\[
\bigg| e^{-is \varphi}(-h^2\Delta_g-(hs)^2)e^{is \varphi}\bigg( \sum_{j=0}^{N(h)} h^k a_k\bigg)\bigg|\le \mathcal{O}(1) e^{-\frac{1}{C_1 h}}, \quad C_1>0,
\]
for all $0<h\le 1$.  Note that we also have
\[
\bigg| \sum_{j=0}^{N(h)} h^k a_k\bigg|\le C\sum_{k=0}^{N(h)} e^{-k}\le C\frac{e}{e-1}. 
\]
In the estimate above we used $k\leq N(h)=[1/(heC)]$ and $\abs{a_k}\leq C^{k+1}k^k$, which holds since $a$ is classical analytic symbol.

Let $\chi\in C_0^\infty(\R^{n-1})$ be such that $0\le \chi\le 1$, $\chi=1$ for $|y|\le 1/4$ and $\chi=0$ for $|y|\ge 1/2$. 
In view of \eqref{eq_3_0_quasimode} we set 
\begin{equation}\label{eq:v_defined}
v(t,y;h)=h^{-\frac{(n-1)}{4}}e^{is\varphi(t,y)}a(t,y;h), \quad a(t,y;h)= \bigg(\sum_{k=0}^{N(h)} h^k a_k\bigg)\chi\bigg(\frac{y}{\delta'}\bigg).
\end{equation}
Here $\delta'>0$ is chosen sufficiently small so that $\chi(y/\delta')$ is zero outside the set where we have constructed the functions $\varphi$ and $a_k$.
Since $\text{Im}\,\varphi(t,y)\ge \frac{|y|^2}{C}$, $C>0$, the cutoff function $\chi$ does not destroy the exponential smallness of the error, and we see that $v$ satisfies after an arbitrarily small decrease of the real domain $U$, 
\begin{equation}
\label{eq_4_3}
\|v\|_{L^2(U)}\asymp 1, \quad \|(-h^2\Delta-(hs)^2)v\|_{L^2(U)}=\mathcal{O}(e^{-\frac{1}{Ch}}), \quad C>0, 
\end{equation}
as $h\to 0$. Here for the first bound,  we use the fact that $a_0\ne 0$.

\subsection{Gluing the local quasimodes together} Let us now return to the open Fermi cover $(U_j)_{j=0}^{N+1}$ of $\gamma([-T_1-\varepsilon, T_2+\varepsilon])$, replacing it if necessary by a slightly smaller relatively compact subcover. We have constructed  $\varphi_0$ real analytic in $U_0$ solving the Cauchy problem \eqref{eq_3_3} in a complex neighborhood of $U_0$ so that $\varphi_0(t,0)=t$  for all $t\in I_0$ and $\text{Im}\, \varphi_0''(t,0)>0$ for all $t\in I_0$. In order to construct  $\varphi=\varphi_1$ real analytic in $U_1$, we pick $t_0\in I_0\cap I_1$ and solve in a complex neighborhood of $U_1$,
\[
\begin{cases}
p(x,\varphi'_x(x))=0,\\
\varphi|_{t=t_0}=\varphi_0(t_0,y),\\
\varphi'_x(t_0)=\xi(t_0).
\end{cases}
\]
We get $\varphi_1$ such that $\varphi_1=\varphi_0$ near $(t_0,0)\in U_0\cap U_1$,   and thus, by unique continuation, $\varphi_1=\varphi_0$ in  $U_0\cap U_1$, assuming as we may that $U_0\cap U_1$ is connected.
Note that in view of \eqref{eq_3_17_new_0}, we have $\varphi_1(t,0)=t$ for all $t\in I_1$. Continuing in this way, we obtain $\varphi_j$ real analytic in $U_j$, $0\le j\le N+1$, such that $\varphi_j=\varphi_{j+1}$ in $U_j\cap U_{j+1}$ and \eqref{eq_3_4} holds for every $\varphi_j$. 

Next let $a^{(0)}(t,y;h)$ be an elliptic classical analytic symbol in a complex neighborhood of $U_0$ obtained by solving \eqref{eq_new_4_1_1}. To get $a^{(1)}(t,y;h)$, we solve the sequence of transport equations \eqref{eq_new_4_1_1} with $\varphi=\varphi_1$ and with $a^{(1)}|_{t=t_0}=a^{(0)}|_{t=t_0}$. Thus, by uniqueness and analytic continuation, $a^{(1)}=a^{(0)}$ in $U_0\cap U_1$. Continuing in the same way, we get $v_0,v_1,\dots, v_{N+1}$ such that 
\begin{equation}
\label{eq_equal_quasi}
v_j=v_{j+1}\quad\text{in}\quad U_j\cap U_{j+1}. 
\end{equation}

Let $\chi_j=\chi_j(t)\in C^\infty_0(I_j)$ be such that $\sum_{j=0}^{N+1} \chi_j=1$ near $[-T_1-\varepsilon, T_2+\varepsilon]$, and define our quasimode $v$ globally by 
\[
v=\sum_{j=0}^{N+1}\chi_j v_j. 
\]

Let $p_1,\dots,p_R\in X^\mathrm{int}$ be the distinct points where the geodesic $\gamma$ self-intersects, and let $-T_1< t_1<\dots<t_{N}<T_2$  be the times of self-intersections. Let $V_1,\dots, V_R$ be small neighborhoods in $X$ around $p_j$, $j=1,\dots, R$. Then choosing $\delta'$ in~\eqref{eq:v_defined} small enough we obtain an open cover of a neighborhood of $\gamma[-T_1,T_2]$ in $\hat X$, 
\begin{equation}
\label{eq_open_cover_sup}
\supp(v(\ccdot ; h))\cap X\subset (\cup_{j=1}^R V_j)\cup (\cup_{k=1}^S W_k),
\end{equation}
where in each $V_j$, the quasimode is a finite sum,
\begin{equation}
\label{eq_rep_vs_1}
v(\ccdot; h)|_{V_j}=\sum_{l: \gamma(t_l)=p_j} v_{l}(\ccdot; h),
\end{equation}
and in each $W_k$ (where there are no self-intersecting points), in view of \eqref{eq_equal_quasi}, there is some $l(k)$ so that the quasimode is given by 
\begin{equation}
\label{eq_rep_vs_2}
v(\ccdot; h)|_{W_k}=v_{l(k)}(\ccdot; h).
\end{equation}

Finally, the bounds \eqref{eq_prop_gaussian_1}  follows from the bounds \eqref{eq_4_3}, and the representations \eqref{eq_rep_vs_1} and \eqref{eq_rep_vs_2} of $v$. This completes the proof of Theorem \ref{thm_main_1}.

\section{Some preliminary results about geodesics}

\label{sec_geodesics}

Let $(X,g)$ be a  compact Riemannian manifold of dimension $n\ge 2$ with smooth boundary, contained in an open real analytic manifold $(\hat X, g)$ of the same dimension with $g$ real analytic in $\hat X$. First we have the following analog of \cite[Lemma 2.1]{DKuLLS_2018}, established in this work in the smooth case. 

\begin{lem}
\label{lem_geodesics_1}
Let $\alpha_0=(x_0,\xi_0)\in S^*X^{\mathrm{int}}$ and let $\zeta_1,\zeta_2\in S^*_{x_0}X^{\mathrm{int}}$ be such that 
\begin{equation}
\label{eq_6_1}
\zeta_1+\zeta_2=t_0\xi_0
\end{equation}
for some $0<t_0<2$. Then there exists a neighborhood $U$ of $\alpha_0$ in $S^*X^{\mathrm{int}}$ and a real analytic map
\[
I:U\to S^*X^{\mathrm{int}}\times S^*X^{\mathrm{int}}, \quad (x,\xi)\mapsto (x,\omega_1(x,\xi))\times (x,\omega_2(x,\xi))
\]
such that 
\begin{equation}
\label{eq_6_2}
I(x_0,\xi_0)=(x_0,\zeta_1)\times (x_0,\zeta_2),
\end{equation}
and 
\begin{equation}
\label{eq_6_3}
\omega_1(x,\xi)+\omega_2(x,\xi)=t_0\xi, \quad (x,\xi)\in U. 
\end{equation}
\end{lem} 

\begin{proof} We follow the proof of  \cite[Lemma 2.1]{DKuLLS_2018}  with minor changes in the real analytic setting, and the argument is presented here only for the convenience of the reader. 

Let $x_1, \dots, x_n$ be real analytic local coordinates on $X^{\mathrm{int}}$ centered at $x_0$ such that 
$G(0)=1$. Here $G = G(x) = (g^{jk})$ is the co-metric tensor. It follows from \eqref{eq_6_1} upon taking the scalar product with $\zeta_1$, $\zeta_2$, and $\xi_0$, that $\zeta_1\cdot \xi_0=\zeta_2\cdot\xi_0$ and $t_0=2 \zeta_1\cdot \xi_0$.  Similarly, if  \eqref{eq_6_3} holds, then $t_0=2G(x)\omega_1(x,\xi)\cdot \xi$, and we therefore should have 
\begin{equation}
\label{eq_6_4}
G(x)\omega_1(x,\xi)\cdot \xi=\zeta_1\cdot \xi_0.
\end{equation}

Furthermore,  if \eqref{eq_6_3} is valid, then 
\[
\omega_2(x,\xi)= 2 (\zeta_1\cdot \xi_0)\xi -  \omega_1(x,\xi).
\]
Note that this implies that $|\omega_2(x,\xi)|_{G(x)}=1$, provided that \eqref{eq_6_4} holds, and therefore, we only need to determine $\omega_1(x,\xi)\in S^*X^{\mathrm{int}}$ depending analytically on $(x,\xi)$ such that $\omega_1(x_0,\xi_0) = \zeta_1$ and \eqref{eq_6_4} is valid.

To that end, let us set 
\[
\zeta(x)=\frac{\zeta_1}{\sqrt{G(x)\zeta_1\cdot \zeta_1}}. 
\]
Then $\zeta(x)$ is real analytic in $x$ and $\zeta(0)=\zeta_1$.   Let 
\[
\tilde \omega_1(x,\xi)=\zeta(x)+\alpha(x,\xi)\xi,
\]
for $(x,\xi)\in \text{neigh}((x_0,\xi_0), S^*\R^n)$, with some $\alpha=\alpha(x,\xi)$ to be chosen. We have 
\[
G(x)\tilde \omega_1(x,\xi)\cdot \tilde \omega_1(x,\xi)=1+\alpha^2+2\alpha G(x)\zeta(x)\cdot \xi. 
\]
We set 
\begin{equation}
\label{eq_6_5}
\omega_1(x,\xi)=\frac{\tilde \omega_1(x,\xi)}{\sqrt{G(x)\tilde \omega_1(x,\xi)\cdot \tilde \omega_1(x,\xi)}}= \frac{\zeta(x)+\alpha\xi}{\sqrt{1+\alpha^2+2\alpha G(x)\zeta(x)\cdot \xi}}.
\end{equation}
We would like to find $\alpha$ so that $\alpha(0,\xi_0)=0$ and that \eqref{eq_6_4} holds. The former requirement guarantees that 
$\omega_1(0,\xi_0)=\zeta_1$, and  the latter requirement implies that we should have  
\[
F(x,\xi,\alpha)=0,
\]
where 
\[
F(x,\xi,\alpha):=(1-(\zeta_1\cdot \xi_0)^2)\alpha^2+2 G (x)\zeta(x)\cdot \xi (1-(\zeta_1\cdot\xi_0)^2)\alpha + (G (x)\zeta(x)\cdot \xi)^2- (\zeta_1\cdot\xi_0)^2.
\]
Note that $F$ is real analytic in $(x,\xi,\alpha)\in \text{neigh}((0,\xi_0), S^*\R^n)\times \R$,  $F(0,\xi_0,0)=0$, and $F_\alpha'(0,\xi_0, 0)=2(\zeta_1\cdot \xi_0)(1-(\zeta_1\cdot \xi_0)^2)\ne 0$, as $0<\zeta_1\cdot\xi_0=\frac{t_0}{2}<1$. Thus, by the implicit function theorem, there is a neighborhood $U$ of $(0,\xi_0)$ and a unique real analytic function $\alpha(x,\xi)$ in $U$ such that $\alpha(0,\xi_0)=0$ and $F(x,\xi,\alpha)=0$ if and only if $\alpha=\alpha(x,\xi)$.  
Hence, $\omega_1(x,\xi)$ given in \eqref{eq_6_5} satisfies the conditions of the proposition. This completes the proof. 
\end{proof}

We shall also need the following result. 

\begin{lem}
\label{lem_geodesic_2}
Let $\alpha_0=(x_0,\xi_0)\in S^*X^{int}$. Assume that $\alpha_0$ is generated by an admissible pair of geodesics $\gamma_1(\alpha_0):[-T_1(\alpha_0), T_2(\alpha_0)]\to X$ and $\gamma_2(\alpha_0):[-S_1(\alpha_0), S_2(\alpha_0)]\to X$, where $0<T_1(\alpha_0), T_2(\alpha_0), S_1(\alpha_0), S_2(\alpha_0)<\infty$.  Then there exists a neighborhood $V$ of $\alpha_0$ in $S^*X^{int}$ such that every point $\alpha=(\alpha_x,\alpha_\xi)\in V$ is generated by an admissible pair of geodesics $\gamma_1(\alpha):[-T_1(\alpha), T_2(\alpha)]\to X$ and $\gamma_2(\alpha):[-S_1(\alpha), S_2(\alpha)]\to X$,  which depend real analytically on $\alpha$. 
\end{lem}

\begin{proof}
First we have 
\[
\zeta_1+\zeta_2=t_0\xi_0,
\]
for some $0<t_0<2$, where $\zeta_j=\dot{\gamma}_j(\alpha_0)$ is viewed as a cotangent vector, using the Riemannian duality. By Lemma \ref{lem_geodesics_1}, there exists a neighborhood $U$ of $\alpha_0$ in $S^*X^{\mathrm{int}}$ and a 
real analytic map
\[
I:U\to S^*X^{\mathrm{int}}\times S^*X^{\mathrm{int}}, \quad \alpha=(\alpha_x,\alpha_\xi)\mapsto (\alpha_x,\omega_1(\alpha))\times (\alpha_x,\omega_2(\alpha))
\]
such that 
\[
I(x_0,\xi_0)=(x_0,\zeta_1)\times (x_0,\zeta_2),
\]
and 
\begin{equation}
\label{eq_6_6}
\omega_1(\alpha)+\omega_2(\alpha)=t_0\alpha_\xi. 
\end{equation}
The corresponding unit speed geodesics $\gamma_1(\alpha):[-T_1(\alpha), T_2(\alpha)]\to X$ and $\gamma_2(\alpha):[-S_1(\alpha), S_2(\alpha)]\to X$, $T_1(\alpha), T_2(\alpha), S_1(\alpha), S_2(\alpha)>0$,   such that $\gamma_j(\alpha)(0)=\alpha_x$, $\dot{\gamma}_j(\alpha)(0)=\omega_j(\alpha)$, are non-tangential being small perturbations of the non-tangential geodesics $\gamma_j(\alpha_0)$, $j=1,2$.  Hence, the functions $T_j(\alpha)$ and $S_j(\alpha)$ depend continuously on $\alpha\in U$ and in particularly, they are bounded after an arbitrarily small decrease of $U$. Note also that  \eqref{eq_6_6} implies that $\gamma_1(\alpha)$ and $\gamma_2(\alpha)$ are two distinct geodesics and that  are not reverses of each other.

We claim that there is a neighborhood $\tilde U\subset U$ of $\alpha_0$ such that for all $\alpha\in \tilde U$, we have 
\begin{equation}
\label{eq_6_6_0}
\gamma_1(\alpha)(t)=\gamma_1(\alpha)(0) \quad \Longleftrightarrow\quad  t=0. 
\end{equation}
Indeed, otherwise there exists a sequence $\alpha_k\to \alpha_0$ as $k\to \infty$ and $0\ne t_k\in [-T_1(\alpha_k), T_2(\alpha_k)]$ such that 
\begin{equation}
\label{eq_6_6_1}
\gamma_1(\alpha_k)(t_k)=\gamma_1(\alpha_k)(0),
\end{equation}
for all $k$. 
Assuming, as we may, that $t_k\to t_0$, we get from  \eqref{eq_6_6_1} that $\gamma_1(\alpha_0)(t_0)=\gamma_1(\alpha_0)(0)$. Since the geodesic $\gamma_1(\alpha_0)$ does not self-intersect at $x_0=\gamma_1(\alpha_0)(0)$, we conclude that $t_0=0$.  Since $\gamma_1(\alpha_k)(t_k)\to \gamma_1(\alpha_0)(0)=x_0\in X^{\mathrm{int}}$ as $k\to\infty$,  for all $k$ sufficiently large, we see that $\gamma_1(\alpha_k)(t_k)\in X^{\mathrm{int}}$.  As $X$ is compact, it has a positive injectivity radius $\text{Inj}(X)>0$. Here we have extended $X$ to a closed manifold to speak about the injectivity radius  and the boundary will not cause any problems as $\gamma_1(\alpha_k)(t_k)\in X^{\mathrm{int}}$, for $k$ sufficiently large. Now \eqref{eq_6_6_1} implies that $|t_k|\ge \text{Inj}(X)$ for all $k$ sufficiently large,  which is a contradiction as $t_k\to 0$. Thus, the claim \eqref{eq_6_6_0} follows.  The same is true for the family of geodesics $\gamma_2(\alpha)$ for $\alpha$ in a possibly smaller neighborhood of $\alpha_0$. 

Finally, we claim there is a neighborhood $V\subset \tilde U$ of $\alpha_0$ such that for all $\alpha\in V$, we have 
\begin{equation}
\label{eq_6_6_2}
\gamma_1(\alpha)(t)=\gamma_2(\alpha)(s) \quad \Longrightarrow \quad t=s=0. 
\end{equation}
Indeed, otherwise, there exists $\alpha_k\to \alpha_0$ as $k\to \infty$, and  $t_k\in [-T_1(\alpha_k), T_2(\alpha_k)]$, and $s_k\in  [-S_1(\alpha_k), S_2(\alpha_k)]$ such that 
\begin{equation}
\label{eq_6_7}
\gamma_1(\alpha_k)(t_k)=\gamma_2(\alpha_k)(s_k),
\end{equation}
$t_k\ne 0$, and $s_k\ne 0$, 
for all $k$.  Assuming as we may that $t_k\to t_0$ and $s_k\to s_0$ and passing to the limit in \eqref{eq_6_7}, we obtain that  
\[
\gamma_1(\alpha_0)(t_0)=\gamma_2(\alpha_0)(s_0),
\]
and therefore, as $\gamma_1(\alpha_0)$ and $\gamma_2(\alpha_0)$ are admissible, we get $t_0=s_0=0$. 
Thus, we get 
\begin{align*}
&\gamma_1(\alpha_k)(t_k)=\gamma_2(\alpha_k)(s_k)\to x_0, \\
&(\alpha_k)_x=\gamma_1(\alpha_k)(0)=\gamma_2(\alpha_k)(0)\to x_0,
\end{align*}
as $k\to \infty$. Note that for $k$ sufficiently large, all the points $\gamma_1(\alpha_k)(t_k)$, $\gamma_1(\alpha_k)(0)$, $\gamma_2(\alpha_k)(s_k)$, $\gamma_2(\alpha_k)(0)$ are in the interior of $X$.  Therefore, $|t_k|\ge \text{Inj}(X)$ and $|s_k|\ge \text{Inj}(X)$ for $k$ sufficiently large, as otherwise, the geodesics $\gamma_1(\alpha_k)$ and $\gamma_2(\alpha_k)$ would intersect at a geodesic ball centered at $(\alpha_k)_x$.  This contradicts the fact that $t_k\to 0$ and $s_k\to 0$ as $k\to \infty$, showing the claim. 
Hence, the pair of geodesics $\gamma_1(\alpha)$, $\gamma_2(\alpha)$ is admissible, for all $\alpha\in V$. 
\end{proof}

\section{Analytic families of exponentially accurate Gaussian beam quasimodes}

\label{sec_analytic_family_quasimodes}

When proving Theorem \ref{thm_main_2} below, we  shall need the following consequence of Theorem \ref{thm_main_1}. 
\begin{cor}
\label{cor_thm_main_1}
Let $(X,g)$ be a  compact Riemannian manifold of dimension $n\ge 2$ with smooth boundary, contained in an open real analytic manifold $(\hat X, g)$ of the same dimension with $g$ real analytic in $\hat X$. 
 Let $\alpha_0=(x_0,\xi_0)\in S^*X^{\mathrm{int}}$ and let 
$\gamma_0:[-T_1,T_2]\to X$, $0<T_1, T_2<\infty$, be   a unit speed nontangential geodesic such that $\gamma_0(0)=x_0$, and $\gamma_0$ does not have self-intersections at $x_0$. Let 
$\gamma(\alpha):[-T_1(\alpha),T_2(\alpha)]\to X$, $0<T_1(\alpha), T_2(\alpha)<\infty$, $\alpha=(\alpha_x,\alpha_\xi)\in \text{neigh}(\alpha_0,S^*X^{\mathrm{int}})$, be a real analytic family of  unit speed nontangential geodesics such that $\gamma(\alpha_0)=\gamma_0$, and $\gamma(\alpha)(0)=\alpha_x$. Let $\lambda\in \R$.  Then there is a real analytic family of $C^\infty$  functions $v(x, \alpha;h)$ on $X$, $\alpha\in \text{neigh}(\alpha_0,S^*X^{\mathrm{int}})$,  $0<h\le 1$, and $C>0$ such that $\supp(v(\ccdot, \alpha;h))$ is confined to a small neighborhood of $\gamma(\alpha)([-T_1(\alpha),T_2(\alpha)])$ for each $\alpha$, and 
\begin{equation}
\label{eq_prop_gaussian_1_new}
\| (-h^2\Delta_g-(hs)^2)v(\cdot, \alpha;h)\|_{L^2(X)}=\mathcal{O}(e^{-\frac{1}{Ch}}), \quad \|v(\cdot, \alpha;h)\|_{L^2(X)}\asymp 1, 
\end{equation}
as $h\to 0$, uniformly in $\alpha$.  Here $s=\frac{1}{h}+i\lambda$. The local structure of the family $v(\ccdot,\alpha;h)$ in a neighborhood of $\alpha_x$
is as follows:  
\[
v(x, \alpha ;h)= h^{-\frac{(n-1)}{4}}e^{is \varphi(x,\alpha)}a(x,\alpha;h), 
\]
where
$\varphi(x,\alpha)$ is real analytic in $(x,\alpha)$ for $\alpha \in \text{neigh}(\alpha_0,S^*X^{\mathrm{int}})$ and $|x-\alpha_x|<\frac{1}{c}$, $c>0$,  and $a(x,\alpha;h)$ is an elliptic classical analytic symbol near $(x_0,\alpha_0)$. Furthermore, for $t$ close to $0$ and $\alpha \in \text{neigh}(\alpha_0,S^*X^{\mathrm{int}})$, we have 
\begin{align*}
\varphi(\gamma(\alpha)(t),\alpha)=t, \quad \nabla \varphi(\gamma(\alpha)(t),\alpha)=\dot{\gamma}(\alpha)(t),\\ \emph{\text{Im}}\, (\nabla^2\varphi(\gamma(\alpha)(t)))\ge 0, \quad \emph{\text{Im}}\, (\nabla^2\varphi)|_{ \dot{\gamma}(\alpha)(t)^\perp}> 0. 
\end{align*}
\end{cor}

\begin{proof}
The functions $T_1(\alpha)$ and $T_2(\alpha)$ depend continuously on $\alpha$ in a small neighborhood  of  $\alpha_0$, and shrinking the neighborhood further we may assume that $T_1(\alpha)$, $T_2(\alpha)$ are bounded.  Let $\varepsilon>0$ be such that $\gamma(\alpha)(t)\in 
\hat{X}\setminus X$ and $\gamma(\alpha)(t)$ has no self-intersection for $t\in [-T_1(\alpha)-2\varepsilon, -T_1(\alpha))\cup (T_2(\alpha),T_2(\alpha)+2\varepsilon]$ for all $\alpha\in \text{neigh}(\alpha_0,S^*X^{\mathrm{int}})$. 
This choice of $\varepsilon$ is possible since $\gamma(\alpha)$ are non-tangential and depend smoothly on $\alpha$.  
It follows from \cite[Lemma 7.2]{Kenig_Salo_APDE_2013} that $\gamma(\alpha)|_{[-T_1(\alpha)-\varepsilon,T_2(\alpha)+\varepsilon]}$ self-intersects only at finitely many times $t_j(\alpha)$, $1\leq j \leq N(\alpha)$, with 
\[
-T_1(\alpha)< t_1(\alpha)<\dots <t_{N(\alpha)}(\alpha)< T_2(\alpha).
\]

First we claim that there is $N_0$ such that $N(\alpha)\le N_0<\infty$ for all $\alpha$ in a small neighborhood of $\alpha_0$. 
This follows by inspection of the arguments in the proof of \cite[Lemma 7.2]{Kenig_Salo_APDE_2013}. Indeed, as explained in \cite[Lemma 7.2]{Kenig_Salo_APDE_2013}, if $\gamma(\alpha)(t)=\gamma(\alpha)(s)$ for some $t\ne s$ then $\dot{\gamma}(\alpha)(t)\ne \pm \dot{\gamma}(\alpha)(s)$.  Furthermore, if $r$ is smaller than the injectivity radius of some closed manifold containing a fixed neighborhood of $X \cup \gamma(\alpha)([-T_1(\alpha) - 2\varepsilon, T_2(\alpha) + 2\varepsilon]$ for $\alpha\in \text{neigh}(\alpha_0,S^*X^{\mathrm{int}})$,  then any two distinct geodesic segments of length $\le r$ can intersect in at most one point. Partitioning $[-T_1(\alpha)-2\varepsilon, T_2(\alpha)+2\varepsilon]$ in disjoint intervals $\{J_k\}_{k=1}^{K(\alpha)}$ of length $\le r$, we get an injective map
\begin{equation}
\label{eq_cardinality}
\begin{aligned}
&\{ (t,s)\in [-T_1(\alpha)-2\varepsilon, T_2(\alpha)+2\varepsilon]^2: s<t \text{ and }\gamma(\alpha)(t)=\gamma(\alpha)(s)\}  \\
&\qquad\qquad\qquad\qquad \longrightarrow \{ (k,l)\in \{1,\dots, K(\alpha)\}^2\},\\
&\quad\quad\qquad\qquad (t,s) \mapsto (k,l) \text{ such that } t\in J_k, s\in J_l. 
\end{aligned} 
\end{equation}
Since $T_1(\alpha)$ and $T_2(\alpha)$ are bounded for $\alpha$ in a small neighborhood of $\alpha_0$, we may assume that $K(\alpha)$ is bounded. 
Consequently, the cardinality of the set $\{ (k,l)\in \{1,\dots, K(\alpha)\}^2\}$ is bounded uniformly in $\alpha$. 
%, for $\alpha$ in a small neighborhood of $\alpha_0$, 
The claim follows.

We also set $t_0(\alpha):=-T_1(\alpha)-\varepsilon$ and $t_{N+1}(\alpha):=T_2(\alpha)+\varepsilon$.
An inspection of the proof of \cite[Lemma 3.5]{DKuLS_2016} allows us to conclude that  there exists an open cover $\{(U_j(\alpha),\kappa_j(\alpha))\}_{j=0}^{N(\alpha)+1}$ of $\gamma(\alpha)([-T_1(\alpha)-\varepsilon,T_2(\alpha)+\varepsilon])$ consisting of coordinate neighborhoods $U_j(\alpha)$ and real analytic diffeomorphisms $\kappa_j(\alpha)$, depending real analytically on $\alpha$, such that the following properties hold, 
\begin{itemize}
\item[(i)] $\kappa_j(\alpha)(U_j(\alpha))=I_j\times B$, where $I_j$ are fixed open intervals and $B=B(0,\delta')$ is an open ball in $\R^{n-1}$.  Here $\delta'>0$  can be taken arbitrarily small and the same for each $U_j(\alpha)$, uniformly for $\alpha$ close to $\alpha_0$, 
\item[(ii)] $\kappa_j(\alpha)(\gamma(\alpha)(t))=(t,0)$ for each $t\in I_j$,
\item[(iii)] $t_j$ only belongs to $I_j$ and $\overline{I_j}\cap \overline{I_k}=\emptyset$ unless $|j-k|\le 1$,
\item[(iv)] $\kappa_j(\alpha)=\kappa_k(\alpha)$ on $\kappa_j^{-1}((I_j\cap I_k)\times B)$.
\end{itemize}
In particular,  the open sets $U_j(\alpha)$ are bounded uniformly in $\alpha$ and contain a fixed open set. 

Following the proof of Theorem \ref{thm_main_1}, and making use of the fact that the geodesics $\gamma(\alpha)$ do not have self-intersections at $\alpha_{x}$, for $\alpha$ close to $\alpha_0$, we obtain the statement of Corollary \ref{cor_thm_main_1}.
\end{proof}

\begin{rem}
Let us also note that in general the number of self-intersecting times $N(\alpha)$ need not depend continuously on $\alpha$. To this end, assume that the dimension of the manifold $X$ is $>2$ and that the geodesic $\gamma_0$ in $X$ has a self-intersection at some point $x_1\in \gamma_0((-T_1, T_2))$ so that $x_1 = \gamma_0(t_1) = \gamma_0(t_2)$, $t_1<t_2$. Then one can show that by means of a small perturbation, that one can unwind the loop in the direction orthogonal to the plane spanned by the velocity vectors  $\dot{\gamma}(t_1)$ and $\dot{\gamma}(t_2)$. 

\end{rem}

\section{Construction of  families of harmonic functions based on Gaussian beam quasimodes}

\label{sec_CGO_based_quasi}

Let $(M,g)$ be a transversally anisotropic manifold of dimension $n\ge 3$ with transversal manifold $(M_0,g_0)$, and assume that $M_0^{\mathrm{int}}$ and $g_0|_{M_0^{\mathrm{int}}}$ are real analytic.

First assume, as we may,  that $(M,g)$ is embedded in a compact smooth manifold $(N,g)$ without boundary of the same dimension, and let $U$ be open in $N$ such that $M\subset U$.  Our starting point is the following Carleman estimate for $-h^2\Delta$, established in \cite{DKSU_2009}.

\begin{prop}
Let $\phi$ be a limiting Carleman weight for $-h^2\Delta$ on  $U$. Then for all $0<h\ll 1$, we have
\begin{equation}
\label{eq_Car_for_laplace}
h\|u\|_{L^2(N)}\le C\|e^{\frac{\phi }{h}}(-h^2\Delta)e^{-\frac{\phi}{h}}  u\|_{L^2(N)}, \quad C>0,
\end{equation}
for all $u\in C_0^\infty(M^\mathrm{int})$. 
\end{prop} 
Using a standard argument, see \cite{DKSU_2009}, we convert the Carleman estimate \eqref{eq_Car_for_laplace} into the following solvability result. 

\begin{prop}
\label{prop_solvability}
Let $\phi$ be a limiting Carleman weight for $-h^2\Delta$ on $U$.  If  $h>0$ is small enough, then for any $v\in L^2(M)$, there is a solution $u\in L^2(M)$ of the equation 
\[
e^{\frac{\phi}{h}}(-h^2\Delta)e^{-\frac{\phi}{h}}  u=v \quad \text{in}\quad M^{\mathrm{int}},
\]
which satisfies 
\[
\|u\|_{L^2(M)}\le \frac{C}{h} \|v\|_{L^2(M)}.
\]
\end{prop} 

Now as $M\subset \subset \R\times  M_0^{\mathrm{int}}$, there is a compact Riemannian manifold $\tilde M_0$ of dimension $n-1$ with smooth boundary such that $M\subset\subset \R\times \tilde M_0\subset \subset \R\times M_0^{\mathrm{int}}$. Note that $(M_0^{\mathrm{int}},g_0|_{M_0^{\mathrm{int}}})$ is an open real analytic manifold with real analytic metric, and we can use Corollary \ref{cor_thm_main_1} to construct a real analytic family of Gaussian beam quasimodes along nontangential geodesics on $\tilde M_0$.

Let us write $x=(x_1, x')$ for local coordinates in $\R\times \tilde M_0$. Let 
\[
s=\frac{1}{h}+i\lambda, \quad  \lambda\in \R, \quad \lambda\quad  \text{fixed}. 
\]
Note by~\cite[Lemma 2.9]{DKSU_2009} that $\phi(x)=\pm  x_1$  is a limiting Carleman weight for $-h^2\Delta$ on $U$. 
 We are interested in finding harmonic functions, 
 \begin{equation}
\label{eq_CGO_6}
-\Delta_{ g} u=0\quad \text{in}\quad M^{\mathrm{int}},
\end{equation}
having the form
\[
 u=u(x,\alpha;h)=e^{-s x_1}(v(x',\alpha;h)+r(x,\alpha;h)),
\]
where $v=v(x', \alpha; h)$ is the Gaussian beam quasimode constructed in Corollary \ref{cor_thm_main_1} on the transversal manifold $\tilde M_0$, associated to a nontangential unit speed geodesic $\gamma(\alpha)$ on $\tilde M_0$ depending analytically on $\alpha\in \text{neigh}(\alpha_0, S^*\tilde M_0^{\mathrm{int}})$,  and $r$ is a remainder term.   
Thus,  $ u$  is a solution of \eqref{eq_CGO_6}
 provided that $r$ solves
 \begin{equation}
 \label{eq_cgo_conjug}
 e^{ \frac{ x_1}{h}} (-h^2\Delta_{ g}) e^{-\frac{ x_1}{h}}(e^{-i \lambda x_1}r) =- e^{-i  \lambda x_1} (-h^2\Delta_{g_0}-(hs)^2) v(x',\alpha ;h).
 \end{equation}
Proposition \ref{prop_solvability} and Corollary \ref{cor_thm_main_1} imply that there is $r=r(\ccdot;\alpha;h)\in L^2(M)$ solving \eqref{eq_cgo_conjug} such that 
\[
\|r\|_{L^2(M)}=\mathcal{O}(e^{-\frac{1}{Ch}}), \quad C>0, 
\]
as $h\to 0$, uniformly in $\alpha\in \text{neigh}(\alpha_0, S^*\tilde M_0^{\mathrm{int}})$.

To summarize, we have the following result. 

\begin{prop}
\label{prop_CGO_general_mnfld}  Let $s=\frac{1}{h}+i\lambda$ with $\lambda\in \R$ being fixed.  For all $h>0$ small enough, there are families of harmonic functions $u_1,u_2\in L^2(M)$, i.e. $-\Delta_g u_j=0$ in $M^{\mathrm{int}}$, having the form
\begin{align*}
& u_1(x,\alpha;h)=e^{-s x_1}(v(x',\alpha;h)+r_1(x,\alpha;h)),\\
& u_2(x,\alpha;h)=e^{s x_1}(v(x',\alpha;h)+r_2(x,\alpha;h)),
\end{align*}
where $v=v(\ccdot,\alpha; h)$ is the family of  Gaussian beam quasimodes constructed in Corollary \ref{cor_thm_main_1}  on $\tilde M_0$,  and $r\in L^2(M)$ is such that $ \|r_j\|_{L^2(M)}=\mathcal{O}(e^{-\frac{1}{Ch}})$, $C>0$, as $h\to 0$, uniformly in $\alpha\in \text{neigh}(\alpha_0, S^*\tilde M_0^{\mathrm{int}})$,  $j=1,2$. 
\end{prop}

\section{Proofs of Theorem \ref{thm_main_2} and Theorem  \ref{thm_main}}

\label{sec_proofs_of_main_theorems}

\subsection{Some facts about analytic wave front sets}
We shall rely on the following characterization of the analytic wave front set, which we recall from \cite[Definition 6.1]{Sjostrand_Analytic} for the convenience of the reader. In our applications, we have $m=n-1$.

\begin{deff}
\label{def_wave_front_set}
Let $\alpha_0=(x_0,\xi_0)\in T^*\R^m\setminus\{0\}$, and let $\varphi(x,\alpha)$, $x\in \R^m$, $\alpha=(\alpha_x,\alpha_\xi)\in T^*\R^m\setminus\{0\}$,  be analytic defined in a neighborhood of  $(x_0,\alpha_0)$ such that  
\begin{equation}
\label{eq_5_1}
\varphi(x,\alpha)|_{x=\alpha_x}=0,\quad \varphi'_x(x,\alpha)|_{x=\alpha_x}=\alpha_\xi, 
\end{equation}
and 
\begin{equation}
\label{eq_5_2}
\emph{\text{Im}}\, \varphi(x,\alpha)\ge C_0|x-\alpha_x|^2, \quad x, \alpha \text{ real},
\end{equation}
for some $C_0>0$. 
Let $a(x,\alpha;h)$ be an elliptic classical analytic symbol defined in a neighborhood of  $(x_0,\alpha_0)$, and let $u\in \mathcal{D}'(X)$, where $X\subset \R^m$ is an open set containing $x_0$. We have $\alpha_0 \notin \emph{\text{WF}}_a(u)$ if and only if there is a real neighborhood $U$ of $\alpha_0$ and $C>0$ such that 
\begin{equation}
\label{eq_5_3}
\sup_{\alpha\in U}|Tu(\alpha;h)|\le Ce^{-\frac{1}{Ch}},
\end{equation}
for $0<h\le 1$, where 
\[
Tu(\alpha;h)=
\int e^{\frac{i\varphi(x,\alpha)}{h}}a(x,\alpha;h)\chi(x)\overline{u(x)}dx,
\]
and $\chi\in C^\infty_0(X)$ is supported in a small neighborhood of  $x_0$ and $\chi=1$ near $x_0$. 
\end{deff}

\begin{rem}
It is established in \cite[Proposition 6.2]{Sjostrand_Analytic} that the condition \eqref{eq_5_3} is independent of the choice of $\chi$, $a$, and $\varphi$. 
\end{rem}

\begin{rem}
\label{rem_wave_front_set}
Assume that $\varphi$, $a$, and $u$ satisfy the same conditions as in Definition \ref{def_wave_front_set}, and let $\psi\in C^\infty_0(\R^n)$ be supported in a small neighborhood of $0$ and  $\psi=1$ near $0$. 
We have $\alpha_0 \notin \emph{\text{WF}}_a(u)$ if and only if there is a real neighborhood $\tilde U$ of $\alpha_0$ and $\tilde C>0$ such that 
\begin{equation}
\label{eq_5_3_new_10}
\sup_{\alpha\in \tilde U}|\tilde Tu(\alpha;h)|\le \tilde Ce^{-\frac{1}{\tilde Ch}},
\end{equation}
for $0<h\le 1$, where 
\[
\tilde Tu(\alpha;h)=
\int e^{\frac{i\varphi(x,\alpha)}{h}}a(x,\alpha;h)\psi(x-\alpha_x)\overline{u(x)}dx.
\]
The condition \eqref{eq_5_3_new_10} is independent of the choice of $\psi$. 
\end{rem}

\begin{rem} 
\label{rem_wave_front_set_1}
The condition \eqref{eq_5_1} in Definition \ref{def_wave_front_set} and Remark \ref{rem_wave_front_set} can be replaced by the following
\begin{equation}
\label{eq_5_1_new}
\varphi(x,\alpha)|_{x=\alpha_x}=f(\alpha)\text{ real valued},\quad \varphi'_x(x,\alpha)|_{x=\alpha_x}=t_0\alpha_\xi, 
\end{equation}
for some fixed $t_0>0$. Indeed, we apply Definition \ref{def_wave_front_set} and Remark \ref{rem_wave_front_set} with  $\varphi(x,\alpha)$ replaced by  $\frac{1}{t_0}(\varphi(x,\alpha)-f(\alpha))$ and with $h$ replaced by $h/t_0$. 
\end{rem}

\begin{rem}
Since the wave front set $\text{WF}_a(u)$ is conic, we may restrict the attention in Definition  \ref{def_wave_front_set}  to $\xi_0\in \R^{m}$ such that $|\xi_0|=1$. 
\end{rem}

\subsection{Proof of Theorem \ref{thm_main_2}}

Let $\alpha_0=(x_0',\xi_0')\in S^*M_0^{\mathrm{int}}$ be generated by an admissible pair of geodesics $\gamma_1(\alpha_0):[-T_1(\alpha_0), T_2(\alpha_0)]\to M_0$ and $\gamma_2(\alpha_0):[-S_1(\alpha_0), S_2(\alpha_0)]\to M_0$. As $M\subset \subset \R\times  M_0^{\mathrm{int}}$, there is a compact Riemannian manifold $\tilde M_0$ of dimension $n-1$ with smooth boundary such that $M\subset\subset \R\times \tilde M_0\subset \subset \R\times M_0^{\mathrm{int}}$, and $x_0'\in \tilde M_0^{\mathrm{int}}$. Furthermore, we can choose $\tilde M_0$ so that the geodesics $\gamma_1(\alpha_0)$ and $\gamma_2(\alpha_0)$ are nontangential on $\tilde M_0$, and hence, $\gamma_1(\alpha_0)$ and $\gamma_2(\alpha_0)$ are admissible on $\tilde M_0$.

Then by Lemma \ref{lem_geodesic_2},  there exists a neighborhood $V$ of $\alpha_0$ in $S^*\tilde M_0^{\mathrm{int}}$ such that every point $\alpha=(\alpha_{x'},\alpha_{\xi'})\in V$ is generated by an admissible pair of geodesics $\gamma_1(\alpha):[-T_1(\alpha), T_2(\alpha)]\to \tilde M_0$ and $\gamma_2(\alpha):[-S_1(\alpha), S_2(\alpha)]\to \tilde M_0$,  which depend real-analytically on $\alpha$. Thus, for all $\alpha\in V$, we have 
\begin{equation}
\label{eq_5_3_2}
\gamma_1(\alpha)(0)=\gamma_2(\alpha)(0)=\alpha_{x'},
\end{equation}
\begin{equation}
\label{eq_5_3_3}
\dot{\gamma}_1(\alpha)(0)+ \dot{\gamma}_2(\alpha)(0)=t_0\alpha_{\xi'},
\end{equation}
for some $0<t_0<2$ fixed, $\gamma_1(\alpha)$, $\gamma_2(\alpha)$ do not have self-intersections at $\alpha_{x'}$,  and $\alpha_{x'}$ is the only point where $\gamma_1(\alpha)$ and $\gamma_2(\alpha)$ intersect, for all $\alpha\in V$.

Let $s_1=\frac{1}{h}+i\lambda$ and $s_2=\frac{1}{h}$, where $\lambda\in \R$. Applying Corollary \ref{cor_thm_main_1}, we get  $v_j(x',\alpha;h)$, $j=1,2$,  Gaussian beam quasimodes on $\tilde M_0$, associated to $\gamma_j(\alpha)$  on $\tilde M_0$,  depending real analytically on $\alpha\in V$  such that 
\begin{equation}
\label{eq_5_3_3_1}
 \|v_j (\ccdot,\alpha;h)\|_{L^2(M)}\asymp 1, \quad \| (-h^2\Delta_{g_0}-(hs)^2)v_j(\ccdot,\alpha;h)\|_{L^2(M)}=\mathcal{O}(e^{-\frac{1}{Ch}}), 
\end{equation}
as $h\to 0$, for some $C>0$, uniformly in $\alpha\in  V$.

An application of Proposition \ref{prop_CGO_general_mnfld} gives harmonic functions on $M$ having the form
\begin{equation}
\label{eq_5_4}
\begin{aligned}
&u_1(x,\alpha;h)=e^{-s_1x_1}(v_1(x',\alpha;h)+r_1(x,\alpha;h)),\\
&u_2(x,\alpha;h)=e^{s_2x_1}(v_2(x',\alpha;h)+r_2(x,\alpha;h)),
\end{aligned}
\end{equation}
where  
\begin{equation}
\label{eq_5_5}
 \|r_j\|_{L^2(M)}=\mathcal{O}(e^{-\frac{1}{Ch}}),\quad  C>0,
\end{equation} 
as $h\to 0$, uniformly in $\alpha\in  V$.

Substituting the harmonic functions $u_1$ and $u_2$ given by \eqref{eq_5_4} into \eqref{eq_int_orthog}, we get 
\begin{equation}
\label{eq_5_6}
\int_M fe^{-i\lambda x_1} (v_1(x',\alpha;h)+r_1)(v_2(x',\alpha;h)+r_2)dV_g=0.
\end{equation}
Using \eqref{eq_5_5} and \eqref{eq_5_3_3_1}, we see that 
\begin{equation}
\label{eq_5_6_1}
\int_M fe^{-i\lambda x_1} v_1(x',\alpha;h)v_2(x',\alpha;h)dV_g=\mathcal{O}(e^{-\frac{1}{Ch}}),\quad  C>0,
\end{equation}
uniformly in $\alpha\in V$. Let us extend $f\in L^\infty(M)$ by zero to $(\R\times M_0)\setminus M$ and set 
\[
\hat f(\lambda, x')=\int_{-\infty}^\infty e^{-i\lambda x_1} f(x_1,x')dx_1
\]
for the Fourier transform with respect to $x_1$.  Using the fact that $dV_g=dx_1dV_{g_0}$, we obtain from  \eqref{eq_5_6_1}  that 
\begin{equation}
\label{eq_5_7}
\int_{\tilde M_0} \hat f(\lambda,x') v_1(x',\alpha;h)v_2(x',\alpha;h)dV_{g_0}=\mathcal{O}(e^{-\frac{1}{Ch}}),\quad  C>0,
\end{equation}
uniformly in $\alpha\in V$. Recalling that the geodesics $\gamma_1(\alpha)$ and $\gamma_2(\alpha)$ intersect at $\alpha_{x'}$ only and that   
\[
\supp(v_j(\ccdot,\alpha;h))\subset \text{small }\text{neigh}(\gamma_j(\alpha)), \quad j=1,2, 
\]
we conclude from \eqref{eq_5_7} that 
\begin{equation}
\label{eq_5_8}
\int_{\text{neigh}(\alpha_{x'}, M_0)} \hat f(\lambda,x') v_1(x',\alpha;h)v_2(x',\alpha;h)\sqrt{g_0(x')}dx'=\mathcal{O}(e^{-\frac{1}{Ch}}), 
\end{equation}
uniformly in $\alpha\in V$. 
%Note that we can take the neighborhood of $\alpha_{x'}$ occurring as the domain of integration in \eqref{eq_5_8} arbitrarily small but fixed.
 Recalling that the geodesics $\gamma_1(\alpha)$, $\gamma_2(\alpha)$ do not have self-intersections at $\alpha_{x'}$, by 
Corollary \ref{cor_thm_main_1}, we have in a small neighborhood of $\alpha_{x'}$, 
\begin{equation}
\label{eq_5_9}
\begin{aligned}
&v_1(x',\alpha;h)=h^{-\frac{(n-2)}{4}}e^{is_1\varphi_1(x',\alpha)}a_1(x',\alpha;h),\\
&v_2(x',\alpha;h)=h^{-\frac{(n-2)}{4}}e^{is_2\varphi_2(x',\alpha)}a_2(x',\alpha;h).
\end{aligned}
\end{equation}
Here $\varphi_j(x',\alpha)$ are real analytic in $(x',\alpha)$ in a region of  the form $\alpha\in V$ and $|x'-\alpha_{x'}|<1/c$, which is an open neighborhood of  $(x_0',\alpha_0)$. Furthermore, $a_j(x', \alpha;h)$ are elliptic classical analytic symbols in a neighborhood of  $(x_0',\alpha_0)$,  $j=1,2$. It follows that the neighborhood of $\alpha_{x'}$ occurring as the domain of integration in \eqref{eq_5_8} can be taken to be fixed and independent of $\alpha$.
We also have for the geodesic parameters $t$ and $s$ near $0$ that %$t$ near $0$ and $s$ near $0$,
\begin{equation}
\label{eq_5_10}
\begin{aligned}
&(\varphi_1)'_{x'}(\gamma_1(\alpha)(t),\alpha)=\dot{\gamma}_1(\alpha)(t), \quad \text{Im}\, ((\varphi_1)''_{x'x'}(\gamma_1(\alpha)(t), \alpha))\ge 0, \\
&\text{Im}\, ((\varphi_1)''_{x'x'}(\gamma_1(\alpha)(t), \alpha)|_{[\dot{\gamma}_1(\alpha)(t)]^\perp})
>0,
\end{aligned}
\end{equation}
and
\begin{equation}
\label{eq_5_11}
\begin{aligned}
&(\varphi_2)'_{x'}(\gamma_2(\alpha)(s),\alpha)=\dot{\gamma}_2(\alpha)(s), \quad \text{Im}\, ((\varphi_2)''_{x'x'}(\gamma_2(\alpha)(s), \alpha))\ge 0, \\
&\text{Im}\, ((\varphi_2)''_{x'x'}(\gamma_2(\alpha)(s), \alpha)|_{[\dot{\gamma}_2(\alpha)(s)]^\perp})
>0.
\end{aligned}
\end{equation}

Now substituting \eqref{eq_5_9} into \eqref{eq_5_8}, we see that 
\begin{equation}
\label{eq_5_12}
\int_{\text{neigh}(\alpha_{x'}, M_0)} e^{\frac{i\varphi(x',\alpha)}{h}}\hat f(\lambda,x') a(x',\alpha;h) dx'=\mathcal{O}(e^{-\frac{1}{Ch}}), \quad h\to 0, 
\end{equation}
uniformly in $\alpha\in V$. Here 
\begin{equation}
\label{eq_5_12_new}
\varphi(x',\alpha)=\varphi_1(x',\alpha)+\varphi_2(x',\alpha)
\end{equation}
is analytic in a neighborhood of $(x_0',\alpha_0)$, 
and 
\[
a(x',\alpha;h)=e^{-\lambda\varphi_1(x',\alpha)}a_1(x',\alpha;h)a_2(x',\alpha;h)\sqrt{g_0(x')}
\]
is an elliptic classical analytic symbol in a neighborhood of $(x_0',\alpha_0)$, since the product of two classical analytic symbols is a classical analytic symbol.

We now claim that the phase function $\varphi(x',\alpha)$ in \eqref{eq_5_12_new} satisfies the conditions \eqref{eq_5_1_new} and \eqref{eq_5_2}. First, in view of  \eqref{eq_5_3_2} and \eqref{eq_prop_gaussian_1_phase_100}, we have 
\begin{equation}
\label{eq_5_13}
\varphi(x',\alpha)|_{x'=\alpha_{x'}}=\varphi_1(\gamma_1(\alpha)(0),\alpha)+\varphi_2(\gamma_2(\alpha)(0),\alpha)=0.
\end{equation}
Using \eqref{eq_5_10}, \eqref{eq_5_11}, and \eqref{eq_5_3_3}, we get 
\begin{equation}
\label{eq_5_14}
\begin{aligned}
\varphi'_{x'}(x',\alpha)|_{x'=\alpha_{x'}}&= (\varphi_1)'_{x'}(\gamma_1(\alpha)(0),\alpha)+ (\varphi_2)'_{x'}(\gamma_2(\alpha)(0),\alpha)\\
&=\dot{\gamma}_1(\alpha)(0)+ \dot{\gamma}_2(\alpha)(0)=t_0\alpha_{\xi'}.
\end{aligned}
\end{equation}
It follows from \eqref{eq_5_13} and \eqref{eq_5_14} that the condition \eqref{eq_5_1_new} holds. Let us now check the condition \eqref{eq_5_2}. To this end, Taylor expanding $\varphi(x',\alpha)$ at $x'=\alpha_{x'}$, we get 
\begin{align*}
\varphi(x',\alpha)=&\varphi(\alpha_{x'},\alpha)+\varphi'_{x'}(\alpha_{x'},\alpha)\cdot (x'-\alpha_{x'})\\
&+\frac{1}{2}\varphi''_{x'x'}(\alpha_{x'},\alpha)(x'-\alpha_{x'})\cdot(x'-\alpha_{x'})+\mathcal{O}(|x'-\alpha_{x'}|^3), 
\end{align*}
and therefore, in view of \eqref{eq_5_13} and \eqref{eq_5_14}, when $x'$ and $\alpha$ are real, we see that 
\[
\text{Im}\, \varphi(x',\alpha)=\frac{1}{2}\text{Im}\, \varphi''_{x'x'}(\alpha_{x'},\alpha)(x'-\alpha_{x'})\cdot(x'-\alpha_{x'})+\mathcal{O}(|x'-\alpha_{x'}|^3).
\]
Hence, the condition \eqref{eq_5_2} is equivalent the following condition, 
\begin{equation}
\label{eq_5_15}
\text{Im}\, \varphi''_{x'x'}(\alpha_{x'},\alpha)>0. 
\end{equation}
Using \eqref{eq_5_10},  \eqref{eq_5_11}, and the fact that the vectors $\dot{\gamma}_1(\alpha)(0)$ and $\dot{\gamma}_2(\alpha)(0)$ are not parallel, we have
\[
\text{Im}\, \varphi''_{x'x'}(\alpha_{x'},\alpha)=\text{Im}\, (\varphi_1)''_{x'x'}(\gamma_1(\alpha)(0),\alpha)+ \text{Im}\, (\varphi_2)''_{x'x'}(\gamma_2(\alpha)(0),\alpha)>0,
\]
showing \eqref{eq_5_15}. Thus, by  Remarks \ref{rem_wave_front_set} and \ref{rem_wave_front_set_1}, in view of \eqref{eq_5_12}, we get 
$\alpha_0\notin \text{WF}_a(\overline{\hat f}(\lambda,\cdot))=\text{WF}_a(\hat{\overline{ f}}(-\lambda,\cdot))$ for all $\lambda\in \R$.  Noting that if \eqref{eq_int_orthog} holds for $f$, it also holds for $\overline{f}$ and $\lambda\in \R$ is arbitrary, we get $\alpha_0\notin \text{WF}_a(\hat f(\lambda,\cdot))$ for all $\lambda\in \R$. This completes the proof of Theorem \ref{thm_main_2}.

\subsection{Proof of Theorem \ref{thm_main}}

Now since every point $(x'_0, \xi'_0)\in S^*M_0^{\mathrm{int}}$ is generated by an admissible pair of geodesics, by Theorem \ref{thm_main_2}, we get $\hat f(\lambda,\cdot)$ is real-analytic in $ M_0^{\mathrm{int}}$ for all $\lambda\in \R$. The fact that  $\hat f(\lambda,\cdot)$ has a compact support in $ M_0^{\mathrm{int}}$  and that $M_0$ is connected implies that  $\hat f(\lambda,\cdot)=0$  for all $\lambda\in \R$, and therefore, $f=0$. This completes the proof of Theorem \ref{thm_main}.

\begin{appendix}

\section{Discussion related to Example \ref{ex_main}}
\label{app_example}

Let $M_0 = \mathbb{S}^1 \times [0,a]$, with $a > 0$, be a cylinder with its usual flat metric $g_0$.  The purpose of this Appendix is to show that every point $(x_0,\xi_0)\in S^*M_0^{\mathrm{int}}$ is generated by an admissible pair of geodesics. 

We have 
\[
T^*(\mathbb{S}^1\times (0,a))\simeq T^*\mathbb{S}^1\times T^*(0,a)\simeq (\mathbb{S}^1\times \R)\times ((0,a)\times\R)\simeq (\mathbb{S}^1\times(0,a))\times \R^2,
\]
and therefore, we may identify $S_{x_0}^*M_0^{\mathrm{int}}$ with the unit circle $\mathbb{S}^1$ in $\R^2\simeq\C$. 

Given $\mathbb{S}^1\ni \xi_0=(\xi_{01},\xi_{02})\simeq \xi_{01}+i\xi_{02}$, we set 
\begin{equation}
\label{eq_app_0}
\xi_1=e^{i\alpha}(\xi_{01}+i\xi_{02})\in \mathbb{S}^1, \quad \xi_2=e^{-i\alpha}(\xi_{01}+i\xi_{02})\in \mathbb{S}^1,
\end{equation}
with $\alpha\in (0,2\pi)$ to be chosen. The geodesics $\gamma_1$ and $\gamma_2$ on $M_0$ such that $\gamma_j(0)=x_0$ and $\dot \gamma_j(0)=\xi_j$, $j=1,2$, are given by 
\begin{align*}
\gamma_1(t)=(x_{01}+\xi_{11}t, x_{02}+\xi_{12}t)\in \R/2\pi\Z\times [0,a],\\
\gamma_2(s)=(x_{01}+\xi_{21}s, x_{02}+\xi_{22}s)\in \R/2\pi\Z\times [0,a].
\end{align*}
The geodesics $\gamma_1$ and $\gamma_2$ are nontangential provided that 
\begin{equation}
\label{eq_app_1}
\begin{aligned}
&\xi_{12}=\text{Im}\, (e^{i\alpha}(\xi_{01}+i\xi_{02}))=\xi_{02}\cos\alpha+\xi_{01}\sin\alpha\ne 0,\\
&\xi_{22}=\text{Im}\, (e^{-i\alpha}(\xi_{01}+i\xi_{02}))=\xi_{02}\cos\alpha-\xi_{01}\sin\alpha\ne 0.
\end{aligned}
\end{equation}
Note that if  $\gamma_1$ and $\gamma_2$ are nontangential then they do not have self-intersections.  

We have in view of \eqref{eq_app_0},
\begin{equation}
\label{eq_app_2}
\xi_1+\xi_2=(2\cos\alpha) \xi_0,
\end{equation}
and therefore, the property (ii) of Definition \ref{def_admissible} follows with $t_0=2\cos\alpha$, provided that 
\begin{equation}
\label{eq_app_3}
0<\cos\alpha<1. 
\end{equation}
Note that $\gamma_1$ and $\gamma_2$ intersect each other if there exist $t$ and $s$ such that 
\begin{equation}
\label{eq_app_4}
\xi_{11}t-\xi_{21}s\in 2\pi\Z,\quad 
\xi_{12}t=\xi_{22}s. 
\end{equation}
Now if we choose $\alpha$ so that 
\begin{equation}
\label{eq_app_5}
|\xi_{11}t-\xi_{21}s|<2\pi, \quad \xi_{12}t=\xi_{22}s, 
\end{equation}
then \eqref{eq_app_4} implies that $\xi_1t=\xi_2s$, and therefore, $|t|=|s|$. In view of \eqref{eq_app_2} and \eqref{eq_app_3}, we get $t=s=0$, and hence, $x_0$ is the only point of intersections of $\gamma_1$ and $\gamma_2$. 

To achive \eqref{eq_app_5}, assuming that \eqref{eq_app_1} holds,  we estimate 
\begin{align*}
|\xi_{11}t-\xi_{21}s|=\frac{|t|}{|\xi_{22}|}|\xi_{22}\xi_{11}-\xi_{21}\xi_{12}|=\frac{|t|}{|\xi_{22}|}|\text{Im}\,( \overline{\xi}_1\xi_2)|= \frac{|t|}{|\xi_{22}|}|\sin(2\alpha)|
\\ \le \frac{a}{|\xi_{12}||\xi_{22}|}|\sin(2\alpha)|
=\frac{a}{|\xi_{02}^2-\sin^2\alpha|}|\sin(2\alpha)|,
\end{align*}
where we use that $0\le x_{02}+\xi_{12}t\le a$, $0\le x_{02}\le a$, and \eqref{eq_app_1}.  Thus, to prove the result, we have to choose $\alpha\in (0,2\pi)$ so that \eqref{eq_app_1}, \eqref{eq_app_3}, and 
\begin{equation}
\label{eq_app_6}
\frac{a}{|\xi_{02}^2-\sin^2\alpha|}|\sin(2\alpha)|<2\pi
\end{equation}
hold. In doing so let us first consider the case when $\xi_{02}\ne 0$. In this case choosing $\alpha>0$ small enough, depending on $a$ and $\xi_{02}$, we see that \eqref{eq_app_1}, \eqref{eq_app_3}, and \eqref{eq_app_6} hold. When $\xi_{02}=0$, we choose $\alpha=\frac{\pi}{2}-\beta$ with $\beta>0$ small enough, depending on $a$. Then \eqref{eq_app_6} becomes
\[
\frac{a}{|\cos^2\beta|}|\sin(2\beta)|<2\pi,
\]
which together with  \eqref{eq_app_1}, \eqref{eq_app_3} hold for such small $\beta$. This completes the proof that every point of  $S^*M_0^{\mathrm{int}}$ is generated by an admissible pair of geodesics.
\end{appendix}

\end{document}